\documentclass[11pt]{amsart}
\usepackage[utf8]{inputenc}
\usepackage{amsmath}
\usepackage{amssymb}
\usepackage{tikz}
\usepackage{color}
\usepackage{xcolor}
\usepackage{float}
\usepackage[caption = false]{subfig}

\usepackage{wasysym}

\usepackage{mathrsfs}
\usepackage{makecell}

\usepackage{geometry}

 \theoremstyle{definition}

 \newtheorem{theorem}{Theorem}[section]%
 \newtheorem{corollary}[theorem]{Corollary}
 \newtheorem{lemma}[theorem]{Lemma}
 \newtheorem{claim}[theorem]{Claim}
 \newtheorem{proposition}[theorem]{Proposition}
 \newtheorem{definition}[theorem]{Definition}
  \numberwithin{equation}{section}

\theoremstyle{remark}

\newcommand{\bit}{\begin{itemize}}
\newcommand{\eit}{\end{itemize}}
\newcommand{\ben}{\begin{enumerate}}
\newcommand{\bens}{\begin{enumerate*}}
\newcommand{\een}{\end{enumerate}}
\newcommand{\eens}{\end{enumerate*}}
\newcommand{\be}{\begin{equation}}
\newcommand{\bes}{\begin{equation*}}
\newcommand{\ee}{\end{equation}}
\newcommand{\ees}{\end{equation*}}
\newcommand{\ba}{\begin{array}}
\newcommand{\ea}{\end{array}}

\newcommand{\abs}[1]{\left|#1\right|}

\newcommand{\dd}{\mathrm{d}}
\newcommand{\dt}{\mathrm{d}t}
\newcommand{\dx}{\mathrm{d}x}
\newcommand{\dy}{\mathrm{d}y}
\newcommand{\co}{\mathrm{\bf co}}
\newcommand{\fel}{\frac{1}{2}}

\newcommand\cI{\mathcal I}
\newcommand\cN{\mathcal N}
\newcommand\cP{\mathcal P}
\newcommand\cS{\mathcal S}

\newcommand\R{\mathbb R}

\newcommand\N{\mathbb N}

\newcommand{\id}{\mathrm{id}}

\newcommand{\ler}[1]{\left( #1 \right)}
\newcommand{\lers}[1]{\left\{ #1 \right\}}

\newcommand\Wp{\mathcal{W}_p}
\newcommand\Wt{\mathcal{W}_2}

\newcommand\Wo{\mathcal{W}_1}
\newcommand\dwpp{d_{\Wp}^p}
\newcommand\dwp{d_{\Wp}}

\newcommand\zo{([0,1])}
\newcommand{\wor}{\mathcal{W}_1(\mathbb{R})}
\newcommand{\dwo}{d_{\mathcal{W}_1}}
\newcommand{\mmn}{M(\mu,\nu)}

\newcommand{\xivmn}{\xi_{v}^{\mu,\nu}}
\newcommand{\xivmnn}{\xi_{v}^{\mu_n,\nun}}
\newcommand{\xihmn}{\xi_{h}^{\mu,\nu}}
\newcommand{\xihmnn}{\xi_{h}^{\mun,\nun}}

\newcommand{\mun}{\mu_n}
\newcommand{\nun}{\nu_n}
\newcommand{\diam}{\mathrm{diam}}
\newcommand{\isemb}{\mathrm{IsEmb}}
\newcommand{\isom}{\mathrm{Isom}}
\newcommand{\dist}{\mathrm{dist}_{\Wp}}
\newcommand{\prob}{\mathcal{P}}

\newcommand*{\bigchi}{\mbox{\Large$\chi$}}

\newcommand\linspan{\operatorname{linspan}}
\def\S{\mathcal{S}}
\newcommand\Lat{\operatorname{Lat}}

\title[Isometric study of $\Wp\zo$ and $\Wp(\R)$]{Isometric study of Wasserstein spaces\\ -- the real line
}                                                                                            
\author[Gy\"orgy P\'al Geh\'er]{Gy\"orgy P\'al Geh\'er}
\address{Gy\"orgy P\'al Geh\'er, Department of Mathematics and Statistics\\ University of Reading\\ Whiteknights\\ P.O.
Box 220\\ Reading RG6 6AX\\ United Kingdom}
\email{G.P.Geher@reading.ac.uk or gehergyuri@gmail.com \newline
\hspace{1.6cm} http://www.math.u-szeged.hu/\~{}gehergy}

\author[Tam\'as Titkos]{Tam\'as Titkos}
\address{Tam\'as Titkos, Alfr\'ed R\'enyi Institute of Mathematics\\ Re\'altanoda u. 13-15.\\
Budapest H-1053\\ Hungary\\ and BBS University of Applied Sciences\\ Alkotm\'any u. 9.\\
Budapest H-1054\\ Hungary}

\email{titkos@renyi.hu \newline http://renyi.hu/\~{}titkos}

\author[D\'aniel Virosztek]{D\'aniel Virosztek}
\address{D\'aniel Virosztek, Institute of Science and Technology Austria \\ Am Campus 1 \\ 3400 Klos\-ter\-neuburg \\ Austria}

\email{daniel.virosztek@ist.ac.at \newline http://pub.ist.ac.at/\~{}dviroszt}

\begin{document}
\subjclass[2010]{Primary: 54E40; 46E27  Secondary: 60A10; 60B05}

\keywords{Wasserstein space, isometric embeddings, isometric rigidity, exotic isometry flow}

\thanks{Geh\'er was supported by the Leverhulme Trust Early Career Fellowship (ECF-2018-125), and also by the Hungarian National Research, Development and Innovation
Office - NKFIH (grant no. K115383).}
\thanks{Titkos was supported by the Hungarian National Research, Development and Innovation Office - NKFIH (grant no. PD128374 and grant no. K115383), by the János Bolyai Research Scholarship of the Hungarian Academy of Sciences, and by the ÚNKP-18-4-BGE-3 New National Excellence Program of the Ministry of Human Capacities.}
\thanks{Virosztek was supported by the ISTFELLOW program of the Institute of Science and Technology Austria (project code IC1027FELL01), by the European Union’s Horizon 2020 research and innovation program under the Marie Sklodowska-Curie Grant Agreement No. 846294, and partially supported by the Hungarian National Research, Development and Innovation Office - NKFIH (grants no. K124152 and no. KH129601).}

\begin{abstract}
Recently Kloeckner described the structure of the isometry group of the quadratic Wasserstein space $\mathcal{W}_2(\R^n)$. It turned out that the case of the real line is exceptional in the sense that there exists an exotic isometry flow. Following this line of investigation, we compute $\isom\ler{\Wp(\R)}$, the isometry group of the Wasserstein space $\Wp(\R)$ for all $p \in [1, \infty)\setminus\{2\}$. We show that $\mathcal{W}_2(\R)$ is also exceptional regarding the parameter $p$: $\Wp(\R)$ is isometrically rigid if and only if $p\neq 2$. Regarding the underlying space, we prove that the exceptionality of $p=2$ disappears if we replace $\R$ by the compact interval $[0,1]$. Surprisingly, in that case, $\Wp\zo$ is isometrically rigid if and only if  $p\neq1$. Moreover, $\Wo\zo$ admits isometries that split mass, and $\isom\ler{\Wo\zo}$ cannot be embedded into $\isom\ler{\wor}$.
\end{abstract}
\maketitle
\tableofcontents

\section{Introduction}
\subsection{Motivation and State of the Art}
Given a complete and separable metric space $X$, one defines its Wasserstein space as the collection of sufficiently concentrated Borel probability measures endowed with a metric which is calculated by means of optimal transport (see the precise definitions in Subsection \ref{s:notions}). 
This notion has strong connections to many flourishing areas in pure and applied mathematics including probability theory \cite{bgl,Butkovsky}, theory of (stochastic) partial differential equations \cite{hairer,navier-stokes}, geometry of metric spaces \cite{LV,VRS,S}, machine learning \cite{m1,m2}, and many more. Besides of these connections, the $p$-Wasserstein space itself is an interesting object, being a measure theoretic analogue of $L^p$ spaces \cite{kloeckner-hausdorff}. 

When working in a metric setting, a natural question arises: \emph{can we compute the group of isometries?} The answer is known for various concrete metric spaces.
However, the problem about the isometric embedding semigroup is usually incomparably more difficult, hence an answer is known only for a few cases.
Classical examples from functional analysis include the Banach--Lamperti theorem \cite{Lamperti} which describes the semigroup of all linear isometric embeddings of  $L^p$ spaces, or the Banach--Stone theorem which describes the group of all linear isometries of the Banach space of all continuous functions over a compact Hausdorff space.
We now recall some more recent examples, concentrating on those where the metric space consists of Borel probability measures. 
The common theoretical importance of all the forthcoming metrics is that they metrise the weak convergence of measures. 
Moln\'ar proved in \cite{molnar-levy} that the space of all Borel probability measures over $\R$ endowed with the L\'evy metric is isometrically rigid, that is, each isometry is a push-forward map induced by an isometry of the underlying space $\R$. This result has been generalised for separable real Banach spaces in \cite{geher-titkos}.
For a more detailed overview of similar results we refer the reader to the survey \cite{virosztek-asm-szeg}, where the case of the Kolmogorov-Smirnov \cite{dolinar-molnar} and the Kuiper distances \cite{geher-kuiper} are also discussed. 

Bertrand and Kloeckner wrote a series of papers \cite{bertrand-kloeckner-hadamard,bertrand-kloeckner-2016,Kloeckner-2010,kloeckner-hausdorff, ultrametric} about the isometry groups of quadratic Wasserstein spaces over various metric spaces.
Here we only recall one of Kloeckner's results \cite[Theorem 1.1]{Kloeckner-2010} in which he described the isometry group of the quadratic Wasserstein space over $\R$, and showed the surprising fact that this space admits so-called exotic isometries whose action is wild in a sense.

We would like to point out that in all of the above results about measures the assumption of bijectivity of the distance preserving maps was crucial in order to obtain these descriptions.
In the present paper we continue our study on the (not necessarily bijective) isometric embeddings of Wasserstein spaces, which we started in the recent paper \cite{gtv}, where we provided a complete description for the case of the discrete metric space.

\subsection{Main results, content of the paper}
Kloeckner proved in \cite{Kloeckner-2010} that the isometric structure of $\mathcal{W}_2(\R)$ is exceptional among $\mathcal{W}_2(\R^n)$ spaces. Namely, there exists an exotic isometry flow of $\mathcal{W}_2(\R)$. Our aim  here is to show that $\mathcal{W}_2(\R)$ is exceptional regarding the parameter $p$ as well. It turns out that exceptionality of $p=2$ disappears if we replace $\R$ by $[0,1]$. Surprisingly, in that case $p=1$ is exceptional. The main result of this paper is to get the full picture in the case of the real line. That is, we compute $\isom\ler{\Wp(\R)}$, the isometry group of the Wasserstein space $\Wp(\R)$ for all $p \in [1, \infty)\setminus\{2\}$, see the table below, where $C_2$ denotes the two-element group.

$$
\begin{array}{|*3{c|}}
\hline {} & \makecell{\isom(\Wp\zo) \\ (\text{Isometric rigidity})} & \makecell{\isom(\Wp\ler{\R}) \\ (\text{Isometric rigidity})}  \\
\hline p=1 & \makecell{C_2\times C_2 \\ (\text{not rigid})}& \makecell{\isom(\R)\\ (\text{rigid})} \\
\hline p>1,\,p\neq2 & \makecell{C_2 \\ (\text{rigid})} & \makecell{\isom(\R) \\ (\text{rigid})}\\
\hline p=2 & \makecell{C_2 \\ (\text{rigid})} & \makecell{\textcolor{gray}{\isom(\R)\ltimes\isom(\R)}\\ \textcolor{gray}{(\text{not rigid})
}}\\
\hline
\end{array}
$$
In fact, besides describing the isometry group, we are able to answer more challenging questions regarding the isometric structure. Below we summarize our results, and briefly sketch the method.\\

Section \ref{s:II} is devoted to handle the case of the interval.
In Subsection \ref{s:wozo} we characterize all isometric embeddings of $\Wo\zo$. 
Using the Harris inequality, we find an extremal metric property which is satisfied exactly for those measures that are either Dirac masses, or convex combinations of two Dirac masses concentrated on $\{0,1\}$. 
Thus we are able to obtain that isometric embeddings are automatically surjective, and that the isometry group is the Klein group. 
Moreover, $\Wo\zo$ admits isometries that split mass, which  is quite unusual for Wasserstein spaces (see the aforementioned papers of Bertrand and Kloeckner). The case $p>1$ is investigated in Subsection \ref{s:wpzo}. Using induction and finding some extremal metric properties, we show that the set of all measures supported on $2^N$ points with equi-distributed weights are left invariant ($N\in\N$). As a consequence, we again have that every isometric embedding is necessarily surjective, and the isometry group consists of the push-forward maps induced by the two isometries of the interval.\\

The case of the real line is investigated in Section \ref{s:III}. The main result of Subsection \ref{s:wor} is that every isometry of $\Wo(\R)$ is implemented by an isometry of $\R$. The main issue here is to find a metric characterization of Dirac masses which we do by examining when the diameter of the metric midpoint set of two measures is minimal. This process naturally leads to the notions of vertical and horizontal bisecting measures, and that of adjacent measures.
As a consequence of our description, we conclude that there are two isometries of $\Wo\zo$ that cannot be extended to an isometry of $\wor$, even though $\Wo\zo$ embeds naturally into $\wor$. Furthermore, the isometry group of $\wor$ does not contain an isomorphic copy of the isometry group of $\Wo\zo$.

In Subsection \ref{s:wpr} we describe the general form of (not necessarily surjective) isometric embeddings of $\Wp(\R)$ for $p>1$ and $p\neq 2$. 
The key ingredients here are the Banach--Lamperti theorem and an abstract Mankiewicz-type extension lemma. 
We show that every isometric embedding is a composition of a push-forward of an isometry of $\R$ and a map which acts as a translation on quantile functions. In Subsection \ref{s:w2r} we take a closer look at Kloeckner's result \cite[Theorem 1.1]{Kloeckner-2010} on $\isom\ler{\mathcal{W}_2(\R)}$, in particular, at its exotic isometry flow. In \cite{Kloeckner-2010}, these exotic isometries were defined explicitly on measures supported on at most two points, and it was proved that there exists a unique extension to the whole space. However, the action of the exotic isometry flow was not given explicitly on general measures.
Our contribution here is to provide a functional analytic description of this action on the whole $\mathcal{W}_2(\R)$ in terms of quantile functions, which involves the well-known Volterra- and a composition operator.

\subsection{Technical preliminaries}\label{s:notions}
The aim of this subsection is to {set} the terminology.

\begin{definition}[Isometric embedding, isometry]\label{defi:isemb-isom}
Let $\big(X,\varrho\big)$ be a metric space. A self-map $f\colon X\rightarrow X$ is called an \emph{isometric embedding} if it preserves the distance, that is, $$\varrho(f(x),f(y))=\varrho(x,y)\qquad (x,y \in X).$$ \emph{Surjective} isometric embeddings are termed \emph{isometries.}
\end{definition}

Note that isometric embeddings acting on $X$ form a unital semigroup which we denote by $\isemb(X)$.
The symbol $\isom(X)$ stands for the group of all isometries. 
We denote by $\prob(X)$ the set of all Borel probability measures on $X$.

\begin{definition}[Coupling measure] \label{defi:coupling}
Let $X$ be a complete and separable metric space and let $\mu,\nu\in\prob(X)$. A Borel probability measure $\pi$ on $X \times X$ is said to be a \emph{coupling} of $\mu$ and $\nu$ if the marginals of $\pi$ are $\mu$ and $\nu$, that is, $\pi\ler{A \times X}=\mu(A)$ and $\pi\ler{X \times B}=\nu(B)$ for all Borel sets $A,B\subseteq X$. The set of all couplings is denoted by $\Pi(\mu,\nu)$.
\end{definition}

By means of couplings, in other words \emph{transport plans}, we can define the $p$-Wasserstein space and the corresponding $p$-Wasserstein distance. For more details we refer the reader to the fundamental works of Villani \cite{villani-book,villani-ams-book}.

\begin{definition}[$p$-Wasserstein space]\label{defi:wass-space}
Let $\big(X,\varrho\big)$ be a complete and separable metric space, and $p\geq 1$ be a parameter. The \emph{$p$-Wasserstein space} $\Wp(X)$ is the set of all $\mu\in\prob(X)$ that satisfy
\bes
\int_X \varrho(x,\hat{x})^p~\mathrm{d}\mu(x)<\infty
\ees
for some (hence all) $\hat{x}\in X$, endowed with the \emph{$p$-Wasserstein distance}
\bes 
\dwp\ler{\mu, \nu}:=\ler{\inf_{\pi \in \Pi(\mu, \nu)} \int_{X \times X} \varrho(x,y)^p~\dd \pi(x,y)}^{\frac{1}{p}}.
\ees
\end{definition}

In words, the $p$-Wasserstein space $\Wp(X)$ is the set of all probability distributions that have finite moment of order $p$, endowed with the $p$-Wasserstein metric $\dwp$. 
We remark that $\dwp$ metrizes the weak convergence and is sensitive to large distances in $X$.
Clearly, the embedding of $X$ as Dirac masses 
\bes
\iota\colon X\to\Wp(X),\qquad \iota(x):=\delta_x
\ees
is distance preserving. Moreover, isometries of the underlying space appear in $\isom(\Wp(X))$ by means of a natural group homomorphism given in \eqref{eq:hashtag} below. Throughout this paper the set of all Dirac masses is denoted by $\Delta(X)$. 

\begin{definition}[Push-forward]
For a measurable map $g\colon X \rightarrow X$ the induced \emph{push-forward map} $g_\# \colon \prob(X)\to\prob(X)$ is defined by
\bes 
\big(g_\#(\mu)\big)(A)=\mu(g^{-1}[A])\qquad(A\subseteq X~\mbox{Borel set}, \; \mu\in\prob(X))
\ees
where $g^{-1}[A]=\{x\in X\,|\, g(x)\in A\}$. 
We call $g_\#(\mu)$ the \emph{push-forward} of $\mu$ with $g$.
If $\psi\in\isom(\R)$, then the push-forward map $\psi_\#$ is an isometry of $\Wp(X)$, and the embedding
\be \label{eq:hashtag}
\#\colon \, \mathrm{Isom} (X) \rightarrow \mathrm{Isom}  \ler{\Wp(X)}, \qquad \psi \mapsto \psi_\#
\ee
is a group homomorphism. Isometries of the form $\psi_\#$ are called \emph{trivial isometries}.
\end{definition}

A special feature of Wasserstein spaces on the real line is that the Wasserstein distance of measures can be calculated by means of their cumulative distribution and quantile functions.

\begin{definition}[Cumulative distribution and quantile functions]
In case of $\big(X,\varrho\big)=\big(\mathbb{R},|\cdot|\big)$, the \emph{cumulative distribution function} of a measure $\mu\in\prob(\mathbb{R})$ is defined as 
\bes
F_\mu\colon \,\mathbb{R} \rightarrow [0,1], \quad x \mapsto F_\mu(x):=\mu\ler{(-\infty,x]}.
\ees
We use the shorthand notation $F_\mu(x-) := \lim_{t\nearrow x} F_\mu(t)$ for the limit from the left.
The \emph{quantile function} of $\mu$ (or the right-continuous generalized inverse of $F_\mu$) is denoted by $F_\mu^{-1}$ and is defined as
\bes
F_\mu^{-1}\colon \, (0,1) \rightarrow\mathbb{R}, \quad y \mapsto F_\mu^{-1}(y):=\sup \left\{ x\in\mathbb{R} \, | \, F_\mu(x) \leq y \right\}.
\ees
In case of $\big(X,\varrho\big)=\big([0,1],|\cdot|\big)$ we shall handle the cumulative distribution and the quantile functions of a $\mu\in\prob\zo$ as $[0,1]\rightarrow [0,1]$ functions. The quantile function in this case is defined by right-continuity at $0$ and it takes the value $1$ at $1$.
\end{definition}

Note that the cumulative distribution function of a $\mu\in\prob\zo$ is monotone increasing, continuous from the right and take the value $1$ at the point $1$. Conversely, any function $F\colon [0,1] \rightarrow [0,1]$ satisfying the above three conditions is the cumulative distribution function of some Borel probability measure on $[0,1]$. Consequently, for any measure  $\mu\in\prob\zo$, the function $F_\mu^{-1}$ is a cumulative distribution function of some measure $\nu \in\prob\zo$, that is, $F_\nu=F_{\mu}^{-1}$.

As was mentioned above, in our setting the $p$-Wasserstein distance can be expressed in terms of the cumulative distribution and quantile functions.
Namely, Vallender proved in \cite{vallender} that
\be \label{eq:vall-central}
\dwo\ler{\mu, \nu} = \int_{-\infty}^\infty \abs{F_\mu(x)-F_\nu(x)}~\dx = \int_0^1 \abs{F_\mu^{-1}(x)-F_\nu^{-1}(x)}~\dx
\ee
for all $\mu,\nu\in\wor$. 
Moreover, Vallender's formula can be generalized in the following way:
\be \label{eq:wp-tav}
\dwp\ler{\mu, \nu} = \ler{\int_0^1 \abs{F_\mu^{-1}-F_\nu^{-1}}^p \dt}^{\frac{1}{p}} \qquad \ler{p>1, \; \mu, \nu \in \Wp(\R)},
\ee
see for instance \cite[Remarks 2.19]{villani-ams-book}.
These two formulae will play an important role in the sequel.

\section{Isometric study of $\Wp\zo$}\label{s:II}
Knowing Kloeckner's result on the exotic isometry flow in $\mathcal{W}_2(\R)$, a natural question arises: how does the isometry group look like when one replaces $\R$ by the compact interval $[0,1]$. Investigating this problem we found on the one hand that exceptionality of $p=2$ disappears, and instead the case $p=1$ becomes exceptional. On the other hand, it turned out that some observations regarding $\Wo\zo$ can be used when dealing with $\wor$, thus we decided to start with the case of $\Wp\zo$ spaces. The aim of this section is to describe the structure of all (not necessarily surjective) distance preserving self-maps of $\Wp\zo$. In fact, we will prove that every isometric embedding is automatically surjective, thus $\isom\ler{\Wp\zo}=\isemb\ler{\Wp\zo}$.

\subsection{$p=1$ -- Isometries splitting mass}
\label{s:wozo}

We start by naming two maps that arise naturally.

\begin{definition}[Reflection in $\Wo\zo$]
The isometry group of $[0,1]$ is isomorphic to $C_2$, and is generated by $r\colon [0,1]\to[0,1]$, $r(x):=1-x$. 
According to \eqref{eq:hashtag}, the push-forward map $r_\#$ is an isometry of $\Wo\zo$ which we call \emph{reflection}.
\end{definition}

\begin{definition}[Flip operation]
The map
\bes
j\colon \, \Wo\zo \rightarrow \Wo\zo, \quad \mu \mapsto j\ler{\mu}, \quad F_{j\ler{\mu}}=F_\mu^{-1}
\ees
is called the \emph{flip operation}. 
The map $j$ is surjective, and thus we see from \eqref{eq:vall-central} that $j\in\mathrm{Isom}(\Wo\zo)$.
\end{definition}

We remark that the flip operation does not send Dirac masses to Dirac masses in general, which is quite unusual among isometries of Wasserstein spaces. Indeed, $$j\ler{\delta_t}=t \delta_0 +(1-t)\delta_1\qquad(0\leq t \leq 1).$$

The essential part of our argument will be to show that the flip operation and reflection generate the semigroup $\isemb(\Wo\zo)$.
As these two maps are bijective, it will follow that $\isemb(\Wo\zo) = \isom(\Wo\zo) = C_2\times C_2$, the Klein group.

When describing the form of isometric embeddings, it is a natural idea to identify those subsets of the space in question that can be characterized by means of certain extremal metric properties. 
Our first observation is that
\bes
\diam(\Wo\zo)=\sup\left\{\int_0^1|F_\mu(t)-F_\nu(t)|~dt\,\Big|\,\mu,\nu\in\Wo\zo\right\}=1,
\ees
and this supremum is attained if and only if $\{\mu,\nu\}=\{\delta_0,\delta_1\}$. Therefore, 
\bes 
\{\delta_0,\delta_1\}=\{\varphi(\delta_0),\varphi(\delta_1)\} \qquad \ler{\varphi\in\isemb(\Wo\zo)}.
\ees
Note also that the triangle inequality $\dwo\ler{\delta_0,\delta_1} \leq \dwo\ler{\delta_0,\mu} + \dwo\ler{\mu,\delta_1}$ is saturated for every measure $\mu \in \Wo\zo$. Indeed,
\bes
\dwo\ler{\delta_0,\mu} +\dwo\ler{\mu,\delta_1}
=\int_{[0,1]} \abs{0-y}~\dd \mu(y)+\int_{[0,1]} \abs{x-1}~\dd \mu (x)=1
\ees
holds for all $\mu\in\Wo\zo$. Let us define the set
\bes
S_t:=\big\{\mu \in\Wo\zo \,\big|\, \dwo\ler{\delta_0, \mu}=t\big\}\qquad(t\in[0,1]),
\ees
which we call the \emph{$t$-slice.}
Clearly, if $\varphi \in \mathrm{IsEmb} \ler{\Wo\zo}$ with $\varphi\ler{\delta_0}=\delta_0$, then $\varphi\ler{S_t} \subseteq S_t$. 

In the next Claim we characterize those elements in $S_t$ that have maximal distance.
The Harris inequality plays an important role in our argument.
We introduce the following notations for functions $F$ and $G$:
$$
	\ler{F \wedge G}(x):=\min\{F(x),G(x)\} \quad\text{and}\quad \ler{F \vee G}(x):=\max\{F(x),G(x)\}.
$$

\begin{claim} \label{claim:s-t-max-dist}
Let $0\leq t\leq 1$. 
The $t$-slice $S_t$ has diameter $2t(1-t)$. 
That is,
\bes\dwo\ler{\rho, \sigma} \leq 2 t (1-t) \qquad \ler{\rho, \sigma \in S_t}
\ees
and 
$$
\dwo\ler{\rho, \sigma} = 2 t (1-t) \;\;\;\iff\;\;\; \{\rho,\sigma\}=\{(1-t) \delta_0+t\delta_1,\delta_t\}.
$$
\end{claim}

\begin{proof}
The statement is trivial if $t=0$ or $t=1$.
Let $t\in(0,1)$ and $\rho, \sigma \in S_t$ be arbitrary but fixed. By \eqref{eq:vall-central} we have 
\be\label{eq:1-t}
\int_0^1 F_\rho(x)~\dx=\int_0^1 F_\sigma(x)~\dx=1-t.
\ee

First, we show that
$\int_0^1\ler{F_{\rho} \wedge F_{\sigma}}(x)~\dx\geq(1-t)^2$
holds. As both $F_\rho$ and $F_\sigma$ are monotone increasing, we have
\bes
\ler{F_\rho(x)-F_\rho(y)}\ler{F_\sigma(x)-F_\sigma(y)}\geq 0 \qquad \ler{x,y \in [0,1]}.
\ees
Consequently,
\be \label{eq:harris}
\int_{0}^1 \int_{0}^1 \ler{F_\rho(x)-F_\rho(y)}\ler{F_\sigma(x)-F_\sigma(y)}~\dy~\dx\geq 0,
\ee
which is equivalent to
\be\label{eq:le}
\int_{0}^1 F_\rho(x) F_\sigma(x)~\dx \geq \int_0^1 F_\rho(x)~\dx \cdot \int_0^1 F_\sigma(x)~\dx.
\ee
As $0 \leq F_\rho(x)\leq 1$ and $0\leq F_\sigma(x) \leq 1$, we have $\ler{F_\rho \wedge F_\sigma} (x) \geq F_\rho(x) F_\sigma(x)$ for all $x \in [0,1]$. 
Therefore, combining this with \eqref{eq:1-t} and \eqref{eq:le} we obtain
\be \label{eq:fo1}
(1-t)^2\leq \int_{0}^1 F_\rho(x) F_\sigma(x)~\dx \leq \int_{0}^1 \ler{F_\rho \wedge F_\sigma} (x)~\dx.
\ee

Second, we prove that inequalities in \eqref{eq:fo1} are equalities  if and only if $$\{\rho,\sigma\}=\{(1-t) \delta_0+t\delta_1,\delta_t\}.$$
Since both $F_\rho$ and $F_\sigma$ are continuous from the right, so are the functions $F_\rho \cdot F_\sigma$ and $F_\rho \wedge F_\sigma$. Therefore, by the equivalence of \eqref{eq:harris} and \eqref{eq:le}, the first inequality in \eqref{eq:fo1} is saturated if and only if we have
\be\label{eq:cica}
	F_\rho(x) = F_\rho(y) \quad\text{or}\quad F_\sigma(x) = F_\sigma(y) \quad \ler{(x,y) \in [0,1) \times [0,1)}.
\ee
Moreover, the second inequality in \eqref{eq:fo1} is saturated if and only if 
$$
	F_\rho(x)F_\sigma(x)=\ler{F_\rho \wedge F_\sigma}(x) \quad \ler{x \in [0,1)},
$$
which means that we have 
\be\label{eq:kutya}
	F_\rho(x) \in \{0,1\} \quad\text{or}\quad F_\sigma(x) \in \{0,1\} \quad \ler{x \in [0,1)}.
\ee
Notice that if any of the distribution functions $F_\sigma$ and $F_\rho$ is constant on $[0,1)$, then by \eqref{eq:1-t} its value must be $1-t$ on $[0,1)$.
Observe also that by \eqref{eq:kutya} at most one of them is constant on $[0,1)$, say $F_\sigma$ is not.
This means that we have $F_\sigma(x) < F_\sigma(y)$ for some $0 \leq x < y<1$. 
However by \eqref{eq:cica}, this implies $F_\rho(\tilde x) = F_\rho(\tilde y)$ for all $0 \leq \tilde x \leq x$ and $y \leq \tilde y<1$, and thus $F_\rho$ must be constant $1-t$ on $[0,1)$, or equivalently, $\rho=(1-t) \delta_0+t\delta_1$. 
It follows from \eqref{eq:kutya} that $F_\sigma(x) \in \{0,1\}$ for all $x \in [0,1)$ which means that $\sigma$ must be a Dirac measure. 
By \eqref{eq:1-t} we conclude $\sigma=\delta_t$.

Now, using \eqref{eq:vall-central}, we get on the one hand that
\be\label{1h}
\dwo\ler{\rho, \sigma}=\int_0^1 \abs{F_\rho(x)-F_\sigma(x)} ~\dx
=\int_0^1 \ler{F_\rho \vee F_\sigma}(x) -\ler{F_\rho \wedge F_\sigma}(x) ~\dx.
\ee
On the other hand, we have
\be\label{2h}
\int_0^1 \ler{F_\rho \vee F_\sigma}(x) +\ler{F_\rho \wedge F_\sigma}(x) ~\dx= \int_0^1 F_\rho(x) +F_\sigma(x) ~\dx =2 (1-t).
\ee
Finally, combining \eqref{1h} and \eqref{2h} with inequality \eqref{eq:fo1}, we conclude that
\bes
\dwo\ler{\rho, \sigma}=2 (1-t) - 2 \int_{0}^1 \ler{F_\rho \wedge F_\sigma} (x) ~\dx\leq2(1-t)-2(1-t)^2=2t(1-t)
\ees
with equality if and only if $\{\rho,\sigma\}=\{(1-t) \delta_0+t\delta_1,\delta_t\}$. The proof is complete.
\end{proof}

We remark here that Claim \ref{claim:s-t-max-dist} roughly speaking describes the shape of $\Wo\zo$, suggesting that the action of a $\varphi\in\isemb(\Wo\zo)$ on $\left\{\delta_0,\delta_{\frac{1}{2}}\right\}$ determines $\varphi$ completely.
This is indeed the case and we make this precise as follows.
\begin{figure}[H]
	\centering
	\includegraphics[scale=0.7]{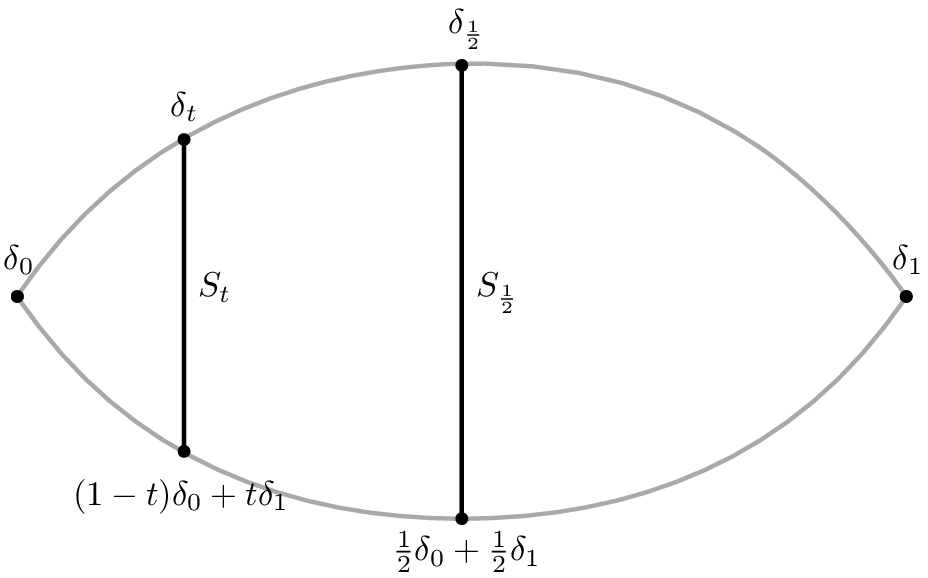}
	\caption{The shape of the Wasserstein space $\Wo\zo$.}\label{fig:shape-w1}
\end{figure}

\begin{claim} \label{claim:identity}
Let $\varphi\colon\Wo\zo\to\Wo\zo$ be an isometric embedding such that $\varphi(\delta_0)=\delta_0$ and $\varphi\left(\delta_{{\frac{1}{2}}}\right)=\delta_{{\frac{1}{2}}}$. Then $\varphi(\mu)=\mu$ for all $\mu\in\Wo\zo$.
\end{claim}

\begin{proof} Using Claim \ref{claim:s-t-max-dist}, we obtain 
$$
	\{\varphi\ler{(1-t) \delta_0+t\delta_1},\varphi\ler{\delta_t}\}=\{(1-t) \delta_0+t\delta_1,\delta_t\} \quad (t\in[0,1]).
$$
In particular, $\varphi\ler{\delta_1}=\delta_1$.
As for all $0<t<1$ we have 
$$
\dwo\ler{\delta_t, \delta_{\frac{1}{2}}}=\abs{t-\tfrac{1}{2}}<\tfrac{1}{2}
\;\;\;\text{and}\;\;\; 
\dwo \ler{(1-t) \delta_0+t\delta_1, \delta_{\frac{1}{2}}}=\tfrac{1}{2},
$$
we get that $\varphi\ler{\delta_t}=\delta_t$. 
Therefore it is enough to show that any measure $\mu \in \Wo\zo$ is completely determined by its distances from Dirac masses. 
This can be seen in the following way: by \eqref{eq:vall-central} we have
\be\label{eq:step2helyett1}
	\dwo\ler{\mu, \delta_t}=\int_0^t F_\mu(x) ~\dx + \int_t^1 \ler{1-F_\mu(x)} ~\dx \quad (t\in [0,1]),
\ee
hence
\be\label{eq:step2helyett2}
	\lim_{h \searrow 0} \frac{\dwo\ler{\mu, \delta_{t+h}}-\dwo\ler{\mu, \delta_t}}{h}=
	\lim_{h \searrow 0} \frac{1}{h}\int_t^{t+h} \ler{2 F_\mu(x)- 1} ~\dx
	=2 F_\mu(t)-1
\ee
holds for all $t \in [0,1)$. 
The proof is done.
\end{proof}

Now we are in the position to present the main result of this subsection.

\begin{theorem} \label{thm:main}
Let $\varphi\in\isemb(\Wo\zo)$, that is,
\bes
\dwo\ler{\varphi(\mu),\varphi(\nu)}=\dwo\ler{\mu, \nu}\qquad(\mu,\nu\in\Wo\zo).
\ees
Then $\varphi\in\{\id_{\Wo\zo},r_{\#},j,r_{\#} j\}$, where $ r_{\#} j= j r_{\#}$.
Consequently, every isometric embedding is surjective, that is,
\bes
\isemb(\Wo\zo)=\isom(\Wo\zo).
\ees
Moreover, this isometry group is isomorphic to the Klein group $C_2\times C_2$.
\end{theorem} 

\begin{proof} 
Clearly, $\varphi(\delta_0)\in\{\delta_0,\delta_1\}$ and it follows from Claim \ref{claim:s-t-max-dist} that 
$$
\varphi\ler{\delta_{{\frac{1}{2}}}} \in \left\{\delta_{{\frac{1}{2}}},{\tfrac{1}{2}}\delta_0+{\tfrac{1}{2}}\delta_1\right\}.
$$ 
Therefore we have four cases to check.
If $\varphi\ler{\delta_0}=\delta_0$ and $\varphi\ler{\delta_{\frac{1}{2}}}=\delta_{\frac{1}{2}}$, then by Claim \ref{claim:identity} $\varphi=\id_{\Wo\zo}$. 
Next, if $\varphi\ler{\delta_0}=\delta_1$ and $\varphi\ler{\delta_{\frac{1}{2}}}=\delta_{\frac{1}{2}}$, then the isometric embedding $r_{\#} \varphi$ sends $\delta_0$ to $\delta_0$ and $\delta_{\frac{1}{2}}$ to $\delta_{\frac{1}{2}}$. 
Consequently, $r_{\#} \varphi=\id_{\Wo\zo}$ and $\varphi=r_{\#}$. Similarly, if $\varphi\ler{\delta_0}=\delta_1$ and $\varphi\ler{\delta_{\frac{1}{2}}}=\frac{1}{2} \delta_{0}+\frac{1}{2} \delta_{1}$, then $j \varphi$ leaves $\delta_0$ and $\delta_{\frac{1}{2}}$ invariant, which implies $\varphi=j$.

\begin{figure}[H]
\centering
	\includegraphics[scale=0.6]{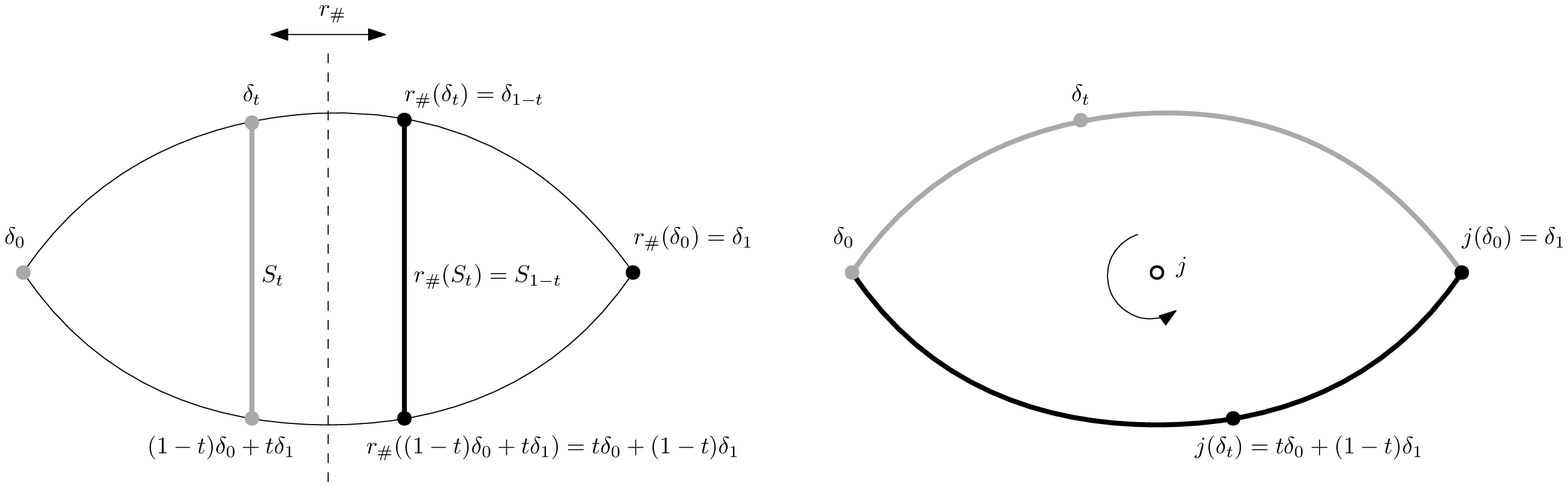}

\caption{{The action of $r_{\#}$ and $j$ on $\Wo\zo$, cf. Figure \ref{fig:shape-w1}.}}
\label{fig:j-action-w1}
\end{figure}

Finally, if $\varphi\ler{\delta_0}=\delta_0$ and $\varphi\ler{\delta_{\frac{1}{2}}}=\frac{1}{2} \delta_{0}+\frac{1}{2} \delta_{1}$, then $r_{\#} j \varphi$ and $j r_{\#} \varphi$ are isometric embeddings leaving both $\delta_0$ and $\delta_{\frac{1}{2}}$ invariant. 
Therefore $\varphi=j r_{\#}=r_{\#} j$. 
\end{proof}

We close this subsection by noting that the metric structure of $\Wo ([a,b])$ is similar to that of $\Wo\zo$ for all $a<b$. Consequently, our method works for any compact interval.
Consider the function 
$\lambda_{a,b}\colon [0,1] \to [a,b], \; \lambda_{a,b}(t) = a+(b-a)t$ and the following push-forward bijection:
$\xi_{a,b}\colon \Wo\zo \to \Wo([a,b]), \; \xi_{a,b}(\mu) = {\lambda_{a,b}}_\#(\mu)$.
Notice that $\varphi\in\mathrm{IsEmb}\ler{\Wo([a,b])}$ holds if and only if 
$(\xi_{a,b})^{-1}\circ\varphi\circ\xi_{a,b} \in \mathrm{IsEmb}\ler{\Wo\zo}$, so $$\mathrm{IsEmb}\ler{\Wo([a,b])} = \mathrm{Isom}\ler{\Wo([a,b])} = C_2\times C_2.$$
In the next subsection we continue by describing all isometric embeddings of the Wasserstein space $\Wp\zo$ for parameters $p>1$.

\subsection{$p>1$ -- Isometric rigidity}
\label{s:wpzo}

{Similarly to the case of $p=1$, it turns out that every isometric embedding is surjective.
However, in contrast to the case $p=1$, we prove \emph{isometric rigidity}, that is, $\isom(\Wp\zo)$ is isomorphic to $\isom\zo=C_2$ for $p>1$.
The main result of this subsection reads as follows.}
{
\begin{theorem} \label{thm:main-2}
Let $p>1$ and let $\varphi\in\isemb(\Wp\zo)$, that is,
\bes
\dwp\ler{\varphi(\mu),\varphi(\nu)}=\dwp\ler{\mu, \nu}\qquad(\mu, \nu \in \Wp\zo).
\ees
Then we have the following two possibilities: 
either $\varphi=\id_{\Wp\zo}$, or $\, \varphi=r_{\#}$.
Consequently, every isometric embedding is surjective, that is,
\bes
\mathrm{IsEmb}\ler{\Wp\zo}=\mathrm{Isom}\ler{\Wp\zo}=C_2.
\ees
\end{theorem}}

{Note that \eqref{eq:wp-tav} and the strict convexity of the $L^p$-norm for $p>1$ implies the following: for any $\mu, \nu \in \Wp\zo$ and $s \in [0,1]$ there exists a unique measure $\gamma_{\mu, \nu}(s) \in \Wp\zo$ such that
\bes
\dwp\ler{\mu, \gamma_{\mu, \nu}(s)}= s \cdot \dwp\ler{\mu, \nu} \;\; \text{and}\;\;\; \dwp\ler{\gamma_{\mu, \nu}(s), \nu}= (1-s) \cdot \dwp\ler{\mu, \nu},
\ees
moreover, $\gamma_{\mu, \nu}(s)$ is defined by the equation
\bes
F_{\gamma_{\mu, \nu}(s)}^{-1}=(1-s) F_{\mu}^{-1}+s F_{\nu}^{-1} \qquad \ler{s \in [0,1]}.
\ees
This is an instance of displacement interpolation (see \cite[Part I. Section 7]{villani-book}), the curve $\gamma_{\mu,\nu}$ is a constant speed geodesic \cite[Chapter 7, (7.2.8)]{AGS}.} {Consequently, for any isometric embedding $\varphi\colon \, \Wp\zo \rightarrow \Wp\zo$ we have the following compatibility equation:
\bes
\varphi\ler{\gamma_{\mu, \nu}(s)}=\gamma_{\varphi\ler{\mu},\varphi\ler{\nu}}(s) \qquad \ler{\mu, \nu \in \Wp\zo, s \in [0,1]}.
\ees
In particular, if $\varphi$ leaves $\mu$ and $\nu$ fixed, then it leaves $\gamma_{\mu,\nu}(s)$ fixed for all $s \in [0,1]$.
We define the \emph{convex hull} $\co(S)$ of a set $S \subseteq \Wp\zo$ as the \emph{closure} of the set of all measures with quantile functions of the form
\bes 
\sum_{j=1}^N \alpha_j F_{\nu_j}^{-1}, \qquad \ler{N \in \N, N\geq1, \, \{(\nu_j,\alpha_j)\}_{j=1}^N\subseteq S\times (0,1],\,\sum_{j=1}^N \alpha_j=1}.
\ees
In other words, $\co(S)$ is the set of those measures whose quantile functions belong to the $L^p$-closed convex hull of quantile functions of measures in $S$.
Now we can generalize the above remark: if an isometric embedding leaves every element of $S$ invariant, then it leaves every element of $\co(S)$ fixed, as well.}
\par
{
Let $M_{-1}:=\lers{\delta_0, \delta_1}$ and $Q_{-1}:=\emptyset.$ Let us introduce
\be \label{eq:qn-def}
Q_n:=\lers{\frac{2k-1}{2^{n+1}} \delta_0+\ler{1-\frac{2k-1}{2^{n+1}}}\delta_1 \, \middle| \, k \in \lers{1, \dots, 2^n}}
\ee
and
\be \label{eq:mn-def}
M_n:=\lers{\frac{1}{2^n} \sum_{j=1}^{2^n} \delta_{a_j} \, \middle| \, 0\leq a_1 \leq \dots \leq a_{2^n}\leq 1}
\ee
for every $n \in \N.$
Note that $M_0=\co\ler{M_{-1} \cup Q_{-1}}=\Delta\zo,$ and it is easy to see by considering the quantile functions that
\be \label{eq:wpzo-step}
M_n=\co \ler{M_{n-1} \cup Q_{n-1}}
\ee
holds for every $n \in \N$ (see Figure \ref{fig:Wpfig}). Indeed, the $n=0$ case is clear, for the $n=1$ case note that $Q_0=\lers{\fel \delta_0+\fel \delta_1},$ and
$$
F^{-1}_{\fel \delta_{a_1}+\fel \delta_{a_2}}
=a_1 F^{-1}_{\delta_1} + \ler{a_2-a_1} F^{-1}_{\fel \delta_0+\fel \delta_1}+(1-a_2) F^{-1}_{\delta_0}.
$$
For general $n \in \N,$ with the convention $a_0:=0$ and $a_{2^n+1}:=1,$ we have
$$
F^{-1}_{\frac{1}{2^n} \sum_{j=1}^{2^n} \delta_{a_j}}
=\sum_{j=0}^{2^n}\ler{a_{j+1}-a_j}F^{-1}_{\frac{j}{2^n} \delta_0 + \ler{1-\frac{j}{2^n}} \delta_1}.
$$
Moreover, we will see that for any $n \in \N,$ the set $Q_n$ is exactly the collection of those measures which are \emph{as far from $M_n$ as possible}. To make this statement precise, we introduce the notation for any measure $\mu \in \Wp\zo$ and any nonempty set $H\subseteq\Wp\zo$
\bes
\dist \ler{\mu,H}:=\inf \big\{d_{\Wp} \ler{\mu, \nu} \, \big| \, \nu \in H\big\}.
\ees
The following Claim regarding the metric structure of $\Wp\zo$ is a key ingredient of the proof of Theorem \ref{thm:main-2}.}
{
\begin{claim} \label{claim:int-1}
For every measure $\mu\in\Wp\zo$ we have 
\be \label{eq:dist-m-n}
\dist \ler{\mu,M_n} \leq \ler{\frac{1}{2}}^{1+\frac{n}{p}},
\ee
and $\dist\ler{\mu,M_n} = \ler{\frac{1}{2}}^{1+\frac{n}{p}}$, if and only if $\mu \in Q_n.$
\begin{figure}[H]
\centering
	\includegraphics[scale=0.6]{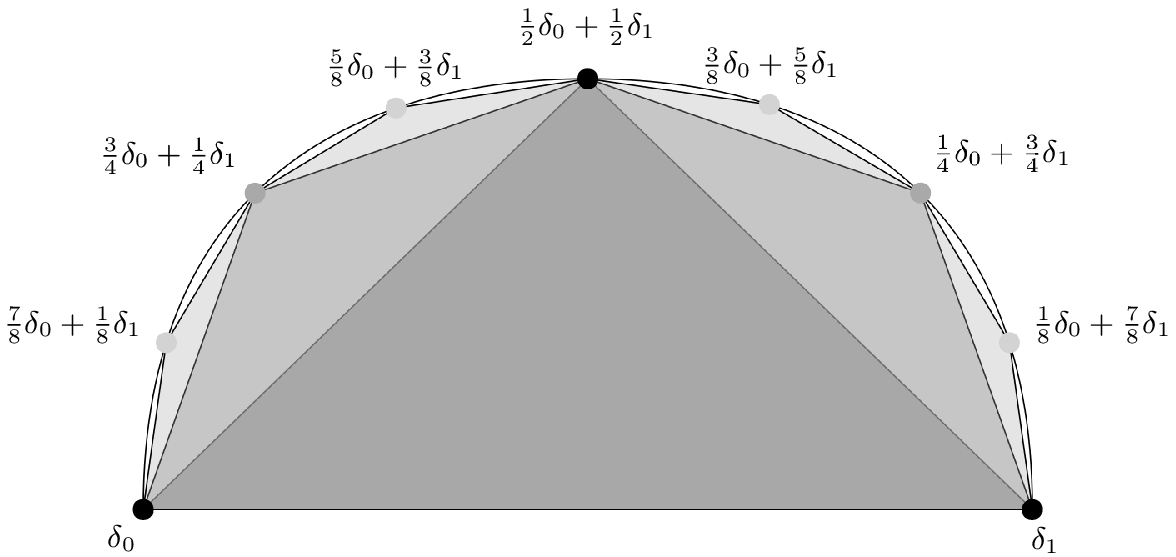}
	\caption{Schematic picture of $\Wp\zo.$}
\label{fig:Wpfig}
\end{figure}
\end{claim}
}
{
\begin{proof}
Let us check the statements of Claim \ref{claim:int-1} for $n=0$ first for the sake of transparency (the left hand side of Figure \ref{fig:m2-tav-max} is intended for this case). The inequality 
\be \label{eq:dist-m1}
\dist \ler{\mu,M_0} \leq {\tfrac{1}{2}}.
\ee
easily  follows from the fact that
\be \label{eq:up-bound}
\dwpp\ler{\mu, \delta_{{\frac{1}{2}}}}=\int_{[0,1]}\abs{x-{\tfrac{1}{2}}}^p \dd \mu(x) \leq \ler{{\tfrac{1}{2}}}^p\qquad(\mu\in\Wp\zo).
\ee
To characterize the case of equality in \eqref{eq:dist-m1}, note that equality in \eqref{eq:up-bound} implies $\mu\ler{\{0,1\}}=1$. 
Consequently, if $\dwp\ler{\mu, \delta_{{\frac{1}{2}}}} = \frac{1}{2}$, then $\mu = (1-\alpha)\delta_0+\alpha \delta_1$ for some $\alpha \in [0,1]$. 
Standard one variable optimization shows that for $\alpha \in (0,1)$ we have
$$
\dist \ler{(1-\alpha)\delta_0+\alpha \delta_1,M_0}
=\dwp \ler{(1-\alpha)\delta_0+\alpha \delta_1, \delta_{t^*(\alpha)}}
$$
where
$$
t^*(\alpha)=\frac{\alpha^{\frac{1}{p-1}}}{\alpha^{\frac{1}{p-1}}+(1-\alpha)^{\frac{1}{p-1}}}.
$$
Another standard one variable optimization argument shows that the maximum of
\begin{equation*}
\begin{split}
\dist&\ler{(1-\alpha)\delta_0+\alpha \delta_1,M_0}\\
&=\ler{
(1-\alpha)\ler{\frac{\alpha^{\frac{1}{p-1}}}{\alpha^{\frac{1}{p-1}}+(1-\alpha)^{\frac{1}{p-1}}}}^p+\alpha \ler{\frac{(1-\alpha)^{\frac{1}{p-1}}}{\alpha^{\frac{1}{p-1}}+(1-\alpha)^{\frac{1}{p-1}}}}^p
}^{\frac{1}{p}}
\end{split}
\end{equation*}
is $\fel,$ and it is taken only at $\alpha=\fel.$
\begin{figure}[H]
\centering
	\includegraphics[scale=0.5]{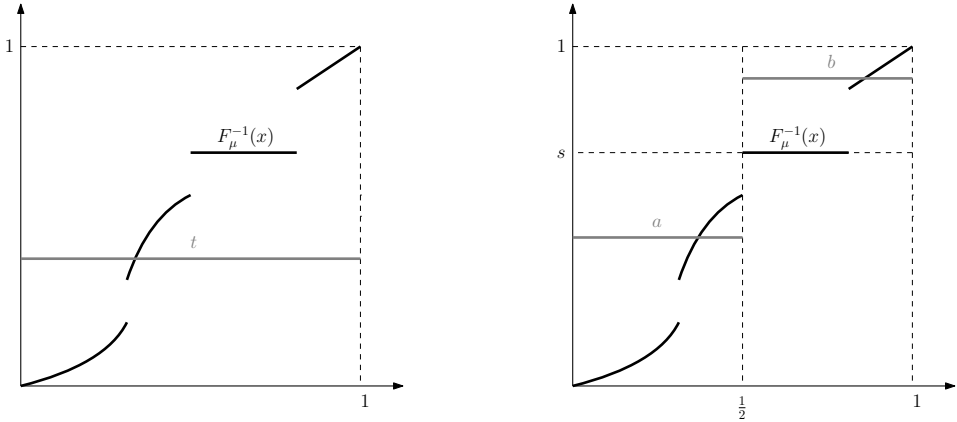}
	\caption{Illustrations for the proof of Claim \ref{claim:int-1}}
\label{fig:m2-tav-max}
\end{figure}
To check Claim \ref{claim:int-1} for any $n \in \N$ (the right hand side of Figure \ref{fig:m2-tav-max} shows the case $n=1$), let us take an arbitrary $\mu \in \Wp\zo$ and introduce $s_0:=0$ and $s_k:=F^{-1}_\mu\ler{\frac{k}{2^n}}$ for $k \in \lers{1, \dots, 2^n}.$ Then

\begin{alignat}{2}
\dist^p\ler{\mu, M_n} &\leq \dwp^p\ler{\mu, \frac{1}{2^n}\sum_{k=1}^{2^n} \delta_{\frac{s_{k-1}+s_k}{2}}}\label{eq:inf-est}\\
&=\sum_{k=1}^{2^n} \int_{\frac{k-1}{2^n}}^{\frac{k}{2^n}}
\abs{F^{-1}_\mu(x)-\frac{s_{k-1}+s_k}{2}}^p \dx
\leq \frac{1}{2^n} \sum_{k=1}^{2^n} \ler{\frac{s_k-s_{k-1}}{2}}^p \label{eq:bound-est}\\
&=\ler{\fel}^{p+n} \sum_{k=1}^{2^n} \ler{s_k-s_{k-1}}^p
\leq \ler{\fel}^{p+n}.\label{eq:conv-est}
\end{alignat}

So \eqref{eq:dist-m-n} is proved, we turn to investigate the case of equality. As $\sum_{k=1}^{2^n} s_k-s_{k-1}=1$ and $p>1$, the inequality in \eqref{eq:conv-est} is saturated if and only if $s_{k^*}-s_{k^*-1}=1$ for some $k^* \in \lers{1, \dots, 2^n}.$ However, if this is the case, the argument presented in the $n=0$ case for the interval $[0,1]$ can be rescaled and applied for the interval $\left[\frac{k^*-1}{2^n}, \frac{k^*}{2^n}\right],$ and we can deduce that \eqref{eq:bound-est} is saturated if and only if $F^{-1}_\mu(x)=0$ for $0\leq x < \fel \ler{\frac{k^*-1}{2^n}+\frac{k^*}{2^n}}=\frac{2 k^*-1}{2^{n+1}}$ and $F^{-1}_\mu(x)=1$ for $\frac{2 k^*-1}{2^{n+1}} \leq x \leq 1,$ that is,
$
\mu=\frac{2k^*-1}{2^{n+1}}\delta_0+\ler{1-\frac{2k^*-1}{2^{n+1}}}\delta_1,
$
and hence $\mu \in Q_n$ --- see \eqref{eq:qn-def}. On the other hand, it is clear that for any $\mu \in Q_n,$
\begin{equation*}
\begin{split}
\dist\ler{\mu, M_n} & =\dist\ler{\frac{2k-1}{2^{n+1}}\delta_0+\ler{1-\frac{2k-1}{2^{n+1}}}\delta_1, M_n}\\
&=\dwp \ler{\frac{2k-1}{2^{n+1}}\delta_0+\ler{1-\frac{2k-1}{2^{n+1}}}\delta_1, \frac{k-1}{2^n}\delta_0 +\frac{1}{2^n} \delta_{\fel} + \frac{2^n-k}{2^n} \delta_1}\\
&=\ler{\frac{1}{2^n}\ler{\frac{1}{2}}^p}^\frac{1}{p}
=\ler{\frac{1}{2}}^{1+\frac{n}{p}}.
\end{split}
\end{equation*}
\end{proof}
}

\begin{proof}[{Proof of Theorem \ref{thm:main-2}}] {We prove the theorem by an induction.
Similarly to the case $p=1,$ every $\varphi\in\isemb(\Wp\zo)$ satisfies $\{\varphi(\delta_0),\varphi(\delta_1)\}=\{\delta_0,\delta_1\}$.
Without loss of generality we assume from now on that
$$
\varphi(\delta_0)=\delta_0\quad\text{and}\quad \varphi(\delta_1)=\delta_1
$$ 
(otherwise we can work with $r_{\#}\varphi$).}
{
In other words, $\varphi$ leaves every element of $M_{-1}$ invariant, and hence by \eqref{eq:wpzo-step}, the same holds for every element of $M_0.$
}
\par
{
Assume that $\varphi(\mu)=\mu$ for every $\mu \in M_n.$ Then $\varphi(\mu) \in Q_n$ for every $\mu \in Q_n,$ because we have seen in Claim \ref{claim:int-1} that $Q_n$ is exactly the collection of those measures which are as far from $M_n$ as possible.
Moreover, every element of $Q_n$ is left invariant by $\varphi,$ because
$$
\dwp\ler{\delta_0, \frac{2k-1}{2^{n+1}} \delta_0+\ler{1-\frac{2k-1}{2^{n+1}}}\delta_1}=\ler{1-\frac{2k-1}{2^{n+1}}}^\frac{1}{p},
$$
and $\varphi\ler{\delta_0}=\delta_0.$ Therefore, every element of $M_{n+1}=\co\ler{M_n \cup Q_n}$ is left invariant by $\varphi.$
}
\par
{
So by the induction we obtained that for every $n \in \N$ and $\mu \in M_n$ the equality $\varphi(\mu)=\mu$ holds. Note that $\bigcup_{n \in \N} M_n$ is a dense subset of $\Wp\zo$, thus every element of $\Wp\zo$ is left invariant by an isometric embedding fixing $\delta_0$. We recall again that if $\varphi(\delta_0)\neq\delta_0$, then $r_{\#}\varphi(\delta_0)=\delta_0$, which forces $r_{\#} \varphi=\id_{\Wp\zo}$, or equivalently, $\varphi=r_{\#}$. The proof is done.}
\end{proof}

As at the end of the previous subsection, one can examine the isometric embeddings of $\Wp([a,b])$ using a map defined very similarly as $\xi_{a,b}$.
One then obtains that all isometric embeddings are bijective, and that there are only the two trivial isometries.

\section{Isometric study of $\Wp(\R)$}\label{s:III}
We have seen that the structure of $\isom\ler{\Wp\zo}$ can be different for different parameters $p$, and that $\isemb\ler{\Wp\zo} = \isom\ler{\Wp\zo}$ for all $p$.
Our next goal is to examine isometries and isometric embeddings of Wasserstein spaces over the real line.
In contrast to the interval case, here it will turn out that $\isom\ler{\Wp(\R)} \subsetneq \isemb\ler{\Wp(\R)}$. 
However, we will also see that the structure of the isometry group can be again different for different parameters $p$, and that the same holds for the semigroup $\isemb\ler{\Wp(\R)}$.
During our investigation, the parameters $p=1$ and $p=2$ have to be handled separately.

As for the $p=2$ case, we recall that the structure of $\isom\ler{\mathcal{W}_2(\R)}$ has been described by Kloeckner in \cite{Kloeckner-2010}. 
In particular, Kloeckner showed that $\mathcal{W}_2(\R)$ admits non-trivial isometries, moreover there exists a so-called exotic flow of isometries that does not even preserve the shape of measures. We will discuss this exotic flow in detail in the last subsection. 

\subsection{$p=1$ -- Isometric rigidity}
\label{s:wor}
The goal of this subsection is to describe the isometry group of $\wor$. Namely, we prove that it admits only trivial isometries. One important difference between the $p=1$ and $p>1$ cases is that the $L^1$ norm is not strictly convex. 
Consequently, in $\wor$ the optimal transport plan between measures is not unique (let alone the geodesic curve). 
Thus Kloeckner's idea of characterizing Dirac masses by means of geodesics cannot be adapted for $p=1$.
Actually, here the main difficulty is to find a metric characterization of Dirac masses.
We start with a definition. 

\begin{definition}[Metric midpoints]
For $\mu,\nu\in\wor$, $\mu\neq\nu$, the following set is called the \emph{metric midpoint set} of $\mu$ and $\nu$:
\begin{align*}
	\mmn:=\Big\{\xi\in\wor\,\Big|\, \dwo(\mu,\xi)=\dwo(\xi,\nu) = \tfrac{1}{2}\dwo(\mu,\nu)\Big\}.
\end{align*}
\end{definition}

We continue with proving a metric property of the metric midpoint set, during which we will identify two special elements of $\mmn$.

\begin{claim}\label{claim:diam}
	For $\mu,\nu\in\wor$, $\mu\neq\nu$ we always have
	\be\label{eq:diamw1}
		\tfrac{1}{2}\dwo(\mu,\nu) \leq \diam\ler{\mmn} \leq \dwo(\mu,\nu).
	\ee
\end{claim}

\begin{proof}
	The second inequality is trivial by the triangle inequality, hence we shall only focus on the first inequality.
	For the sake of brevity, we introduce the notation $D:=\dwo(\mu,\nu)\neq0$.
{By a simple geometric consideration it follows from \eqref{eq:vall-central} that $D$ is exactly the Lebesgue measure of the Borel set}
	$$
		\cS := \left\{ (x,y)\in\R\times[0,1] \,\big|\, F_\mu(x) \leq y \leq F_\nu(x) \; \text{or} \; F_\nu(x) \leq y \leq F_\mu(x) \right\}.
	$$
	Hence, there exist two numbers $h\in(0,1)$ and $v\in\R$ such that the following four sets have Lebesgue measure $\frac{D}{2}$:
	$$
		\cS \cap \left( (-\infty,v)\times [0,1] \right), \;\; \cS \cap \left( (v,\infty)\times [0,1] \right), \;\; \cS \cap \left( \R\times[0,h) \right), \;\; \cS \cap \left( \R\times(h,1] \right).
	$$
	Let us define the sets
	\begin{align*}
		\cS_1 := \cS \cap \left( (-\infty,v)\times [0,h) \right), \;\; \cS_2 := \cS \cap \left( (v,\infty)\times [0,h) \right), \\ 
		\cS_3 := \cS \cap \left( (v,\infty)\times (h,1] \right), \;\; \cS_4 := \cS \cap \left( (-\infty,v)\times (h,1] \right),
	\end{align*}
	and denote the Lebesgue measure of these sets by $\alpha_1, \alpha_2, \alpha_3$ and $\alpha_4$, respectively. 
	By the choice of $v$ and $h$ (see also Figure \ref{fig:beta}) we have 
	$$
		\alpha_1=\alpha_3, \;\; \alpha_2=\alpha_4, \;\;\text{and}\;\; 
		\alpha_1+\alpha_2 = \alpha_3+\alpha_4 = \tfrac{D}{2}.
	$$
	
	From here we consider two cases. 
	First, assume that $\alpha_2=\alpha_4=0$. 
	Suppose for a moment that we have $F_\mu(v) < h$ (the case $F_\nu(v) < h$ is handled similarly). 
	Then by right-continuity we obtain easily that $F_\mu(x) = F_\nu(x)$ for all $v\leq x < \sup \left\{ x\in\mathbb{R} \, | \, F_\mu(x) < h \right\}$. 
	Therefore by monotonicity of $F_\mu$, we can choose $h$ to be $F_\mu(v)$. 
	By doing so, we may assume without loss of generality from now on that 
	$$
	F_\mu(v) \geq h \;\;\;\text{and}\;\;\; F_\nu(v) \geq h.
	$$
	Notice that since $\alpha_4=0$, we must have 
	$$
	\ler{F_\mu\vee F_\nu}(v-) \leq \ler{F_\mu\wedge F_\nu}(v).
	$$
	Hence, the measures $\xi$ and $\eta$ defined in the following way are clearly in $\mmn$ and their distance is obviously $D$, which proves the inequality for this case:
	$$
		F_\xi(x) := \left\{
		\begin{matrix}
			F_\mu(x) & \text{if}\; x<v \\
			F_\nu(x) & \text{if}\; x\geq v \\
		\end{matrix}
		\right.
		\;\;\;\text{and}\;\;\;
		F_\eta(x) := \left\{
		\begin{matrix}
			F_\nu(x) & \text{if}\; x<v \\
			F_\mu(x) & \text{if}\; x\geq v \\
		\end{matrix}
		\right..
	$$
	
	Second, assume that $\alpha_2=\alpha_4 > 0$. Notice that we cannot have both $h \leq F_\mu(v)$ and $h \leq F_\nu(v)$, since that would imply $\cS_2 = \emptyset$. Similarly, having both $F_\mu(v-) \leq h$ and $F_\nu(v-) \leq h$ would imply $\cS_4=\emptyset$. Hence by symmetry we may assume without loss of generality that
	$$
		F_\mu(v) < h < F_\nu(v-).
	$$ 
	Now, we define two measures $\xivmn$ and $\xihmn$ with their cumulative distribution functions:
	\begin{equation}\label{eq:ver-hor1}
		F_{\xivmn}(x) := \left\{
		\begin{matrix}
			F_\mu(x) & \text{if}\; x<v \\
			F_\nu(x) & \text{if}\; x\geq v \\
		\end{matrix}
		\right.
	\end{equation}

	\begin{figure}[H]
	\centering
	\includegraphics[scale = 0.7]{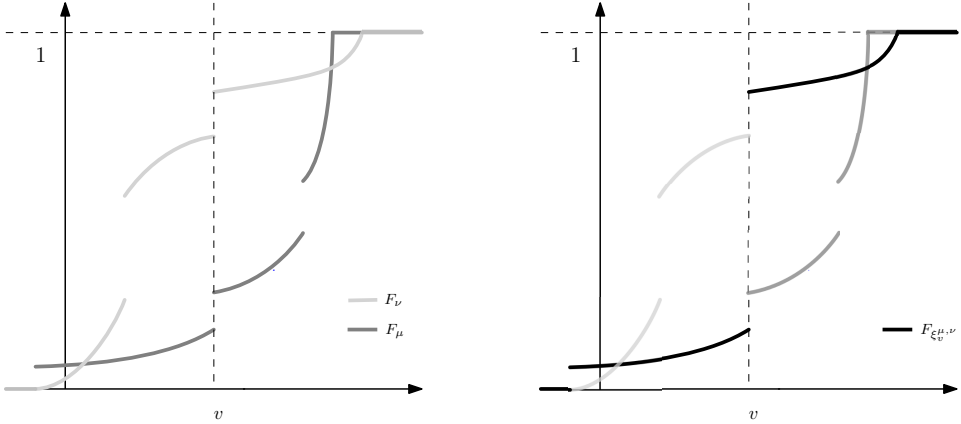}
	\caption{$F_{\xivmn}$ splits the area between $F_{\mu}$ and $F_{\nu}$ vertically.}\label{fig:vebf}
	\end{figure}
	
	and
	\begin{equation}\label{eq:ver-hor2}
		F_{\xihmn}(x) := \left\{
		\begin{matrix}
			F_\nu(x) & \text{if}\; x < F_\nu^{-1}(h) \\
			h & \text{if}\; F_\nu^{-1}(h) \leq x < F_\mu^{-1}(h) \\
			F_\mu(x) & \text{if}\; x\geq F_\mu^{-1}(h) \\
		\end{matrix}
		\right..
	\end{equation}
	
	\begin{figure}[H] 
	\centering
	\includegraphics[scale = 0.7]{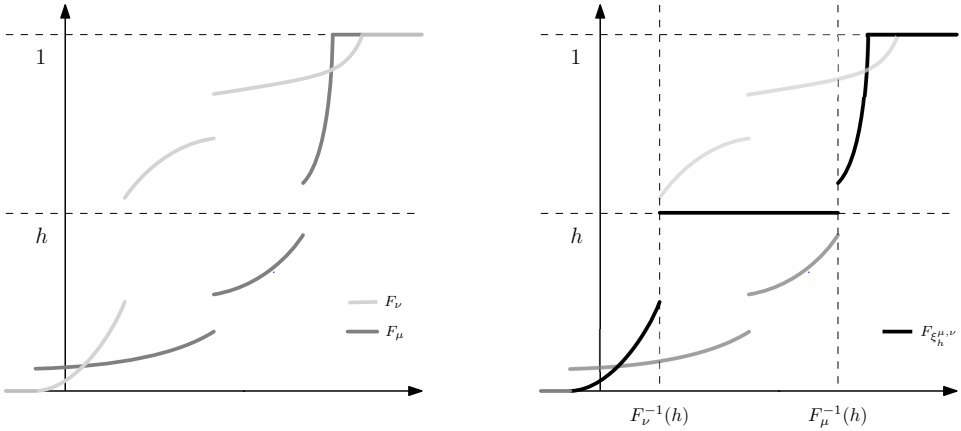}
	\caption{$F_{\xihmn}$ splits the area between $F_{\mu}$ and $F_{\nu}$ horizontally.}\label{fig:hobf}
	\end{figure}
	It is obvious that $F_\nu^{-1}(h) < v < F_\mu^{-1}(h)$, $\xivmn, \xihmn\in\mmn$ and $\dwo(\xivmn,\xihmn) = \alpha_1+\alpha_3$. From here we verify that $\alpha_1=\alpha_3 \geq \alpha_2 = \alpha_4$ by the following geometric observation. 
	We consider the auxiliary rectangle $\left(F_{\nu}^{-1}(h),F_{\mu}^{-1}(h)\right)\times\left(F_{\mu}(v),F_{\nu}(v)\right)$, and split it into four parts using the horizontal and vertical lines corresponding to $h$ and $v$, respectively, see Figure \ref{fig:beta}. 
	Denoting the area of these pieces by $\beta_i$'s in accordance with $\alpha_i$'s ($1\leq i\leq 4)$, we obtain
	\begin{equation}\label{eq:beta-alpha-viszony}
		\beta_1\leq\alpha_1, \;\; \beta_2\geq\alpha_2, \;\; \beta_3\leq\alpha_3, \;\; \text{and} \;\; \beta_4\geq\alpha_4.
	\end{equation}
	But obviously, depending on $h$ we have either $\beta_4\leq\beta_1$, or $\beta_2\leq\beta_3$, which combined with the previous inequalities completes the proof.
\end{proof}

	\begin{figure}[H]
	\centering
	\includegraphics[scale=0.7]{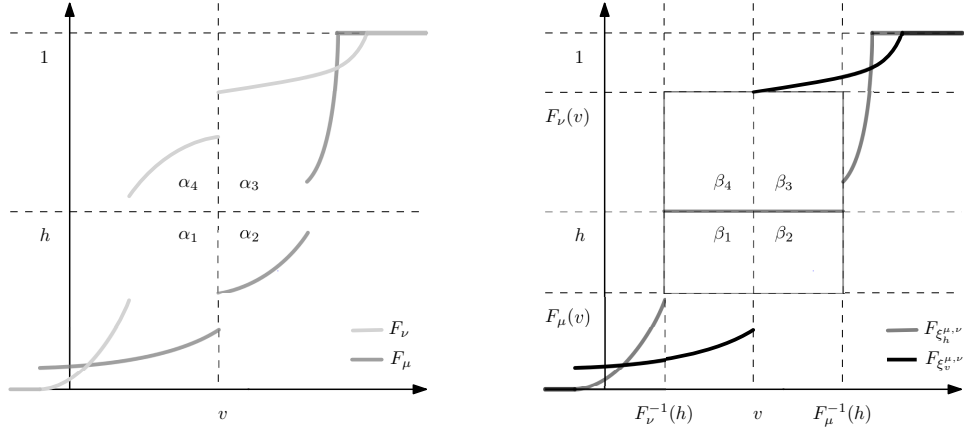}
	\caption{Partitioning the area between the graphs with vertical and horizontal lines, and the auxiliary rectangle.}\label{fig:beta}
	\end{figure}

\begin{definition}[Vertical and horizontal bisecting measures]\label{def:vhbis}
If $\mu,\nu\in\wor$ are measures such that $\alpha_2=\alpha_4 > 0$ with the above defined numbers, then the measures $\xivmn$ and $\xihmn$ defined in \eqref{eq:ver-hor1}--\eqref{eq:ver-hor2} are called the \emph{vertical} and \emph{horizontal bisecting measures} of $\mu$ and $\nu$, respectively.
\end{definition}

We proceed with examining when the first inequality in \eqref{eq:diamw1} becomes an equality.

\begin{definition}[Adjacent measures]
Two different elements $\mu$ and $\nu$ of $\wor$ are said to be \emph{adjacent}, in notation $\mu\sim\nu$, if there exists an interval $(a,b)\subseteq\R$ such that
	\begin{itemize}
	\item[(1)] $\mu|_{\R\setminus\{a,b\}}=\nu|_{\R\setminus\{a,b\}}$ and
	\item[(2)] $\mu\big((a,b)\big)=\nu\big((a,b)\big)=0$.
	\end{itemize}
	Or equivalently,
	\begin{itemize}
 	\item[(1')] $F_{\mu}|_{\R\setminus[a,b)}\equiv F_{\nu}|_{\R\setminus[a,b)}$ and
        \item[(2')] both $F_{\mu}|_{[a,b)}$ and $F_{\nu}|_{[a,b)}$ are constant.
	\end{itemize}
\end{definition}

Observe that for adjacent measures we have $\alpha_2=\alpha_4>0$, hence the vertical and horizontal bisecting measures are defined by Definition \ref{def:vhbis}.

\begin{claim}\label{claim: sim characterization}
For any $\mu,\nu\in\wor, \mu\neq\nu$ the following statements are equivalent
\begin{itemize}
    \item[(i)] $\mathrm{diam}\ler\mmn=\frac{1}{2}\dwo(\mu,\nu)$,
    \item[(ii)] $\mu\sim\nu$.
\end{itemize}
Moreover, if $\mu\sim\nu$, then the diameter is attained only for the pair $\left\{\xivmn,\xihmn\right\}$.
\end{claim}

\begin{proof} We continue to use the notations of Claim \ref{claim:diam}. First, we prove the direction (i)$\Longrightarrow$(ii).
	As $\tfrac{D}{2} = \mathrm{diam}\ler\mmn \geq \alpha_1+\alpha_3 \geq \alpha_2+\alpha_4$, we immediately obtain $\alpha_1=\alpha_2=\alpha_3=\alpha_4=\tfrac{D}{4}$. Combining this with \eqref{eq:beta-alpha-viszony} gives 
	$$
		\beta_1 \leq \tfrac{D}{4}, \;\; \beta_2 \geq \tfrac{D}{4}, \;\; \beta_3 \leq \tfrac{D}{4} \;\;\text{and}\;\; \beta_4 \geq \tfrac{D}{4},
	$$
	from which, by simple geometric considerations, we conclude $\beta_i = \tfrac{D}{4} = \alpha_i$ for all $i=1,2,3,4$.
	In particular, $\mu\sim\nu$ follows.
	
	As for the reverse direction (ii)$\Longrightarrow$(i), we only need to observe that
	\begin{equation*}
		F_{\eta}|_{\R\setminus[a,b]}=F_{\mu}|_{\R\setminus[a,b]}=F_{\nu}|_{\R\setminus[a,b]} \qquad (\eta\in\mmn).
	\end{equation*}
	Indeed, elements of $\mmn$ saturate the triangle inequality
	$$
	\dwo(\mu,\nu) = \dwo(\mu,\eta) + \dwo(\eta,\nu).
	$$
	Hence by \eqref{eq:vall-central} we have $F_\mu\wedge F_\nu \leq F_\eta \leq F_\mu\vee F_\nu$.
	Now, we basically reduced the problem to the case of the interval, and thus the argument of Claim \ref{claim:s-t-max-dist} can be applied with a simple rescaling. In such a way one obtains
	\bes
	\dwo(\eta_1,\eta_2) = \int_a^b |F_{\eta_1}(t)-F_{\eta_2}(t)|\dt\leq{\tfrac{1}{2}} D  \qquad \ler{\eta_1,\eta_2\in\mmn}
	\ees
	with equality if and only if $\{\eta_1,\eta_2\}=\{\xi_v^{\mu,\nu},\xi_h^{\mu,\nu}\}$.
\end{proof}

Now, we are in the position to give a metric characterization of Dirac masses.

\begin{claim}\label{claim: dirac characterization}
For a measure $\eta\in\wor$ the following statements are equivalent
\begin{itemize}
    \item[(i)] $\eta\in\Delta(\R)$,
    \item[(ii)] for all $n\in\mathbb{N}$ there are measures $\mun,\nun\in\wor$ such that
    \begin{itemize}
        \item[(a)] $\mun\sim\nun$,
        \item[(b)] $\dwo(\mun,\nun)=n$,
        \item[(c)] $\eta\in\{\xivmnn,\xihmnn\}$.
    \end{itemize}
\end{itemize}
\end{claim}

\begin{proof}
Assume first that $\eta=\delta_t$ for some $t\in\R$. Then the choices $\mun:=\delta_{t-{\frac{1}{2}} n}$ and $\nun:=\delta_{t+{\frac{1}{2}} n}$ ($n\in\N$) obviously satisfy (ii). 
Therefore what remained to show is that it is impossible to have $\eta\notin\Delta(\R)$ and (ii) at the same time.
We shall prove this {by contradiction}, so from now on we assume that $\eta\notin\Delta(\R)$ fulfils (ii).

By definition, for all $n\in\N$ there exists an interval $[a_n,b_n)$ such that
$$
	F_{\mun}|_{\R\setminus[a_n,b_n)} = F_{\nun}|_{\R\setminus[a_n,b_n)} = F_{\eta}|_{\R\setminus[a_n,b_n)},
$$
and that $F_{\mun}|_{[a_n,b_n)}$ and $F_{\nun}|_{[a_n,b_n)}$ are both constants.
Set 
$$
	\alpha_n := \left(F_{\mun} \wedge F_{\nun}\right)(a_n), \quad\text{and}\quad \beta_n := \left(F_{\mun} \vee F_{\nun}\right)(a_n).
$$
Notice that $(\beta_n-\alpha_n)(b_n-a_n) = n$, hence $b_n-a_n\geq n$.

If $\eta = \xihmnn$, then $F_{\eta}|_{[a_n,b_n)}$ is also constant with $F_{\eta}(a_n) = \tfrac{\alpha_n+\beta_n}{2}$.
A simple geometric consideration shows that in this case we have 
$$
	\dwo(\eta,\delta_0) \geq \int_{a_n}^{b_n} |F_\eta(t) - F_{\delta_0}(t)|\dt \geq \tfrac{n}{2}.
$$
Therefore there exists a number $N\in\N$ such that 
$$
\eta = \xivmnn \quad (n \geq N).
$$

Again by definition, for all $n \geq N$ we have that $F_\eta$ is constant $\alpha_n$ on $\left[a_n,\tfrac{1}{2}(a_n+b_n)\right)$, and constant $\beta_n$ on $\left[\tfrac{1}{2}(a_n+b_n),b_n\right)$.
As $\eta$ is not a Dirac mass, we get that there is a maximal positive number $s_n \geq \tfrac{b_n-a_n}{2}$ such that $F_\eta$ is constant on both intervals $$\left[\tfrac{1}{2}(a_n+b_n)-s_n, \tfrac{1}{2}(a_n+b_n)\right)\qquad\mbox{and}\quad\left[\tfrac{1}{2}(a_n+b_n), \tfrac{1}{2}(a_n+b_n)+s_n\right).$$
Therefore there exists an infinite subset $\cN$ of positive integers such that for all $j,k\in\cN, j\neq k$ we have
$$
	\left[a_{j},\tfrac{1}{2}(a_{j}+b_{j})\right) \cap \left[\tfrac{1}{2}(a_{k}+b_{k}),b_{k}\right) = \emptyset
$$
or
$$
	\left[a_{k},\tfrac{1}{2}(a_{k}+b_{k})\right) \cap \left[\tfrac{1}{2}(a_{j}+b_{j}),b_{j}\right) = \emptyset.
$$
Keeping in mind that $b_n-a_n\geq n$ $(n\in\N)$ gives that $\left\{\tfrac{1}{2}(a_{j}+b_{j})\, |\, j\in\cN\right\}$ is a set that clusters at $+\infty$ or $-\infty$.
Suppose it clusters at least at $+\infty$. Then for large enough numbers $j\in\cN$ one easily concludes that
$$
	\dwo(\eta,\delta_0) \geq \int_{a_j}^{\tfrac{1}{2}(a_{j}+b_{j})} \left|F_\eta(t) - F_{\delta_0}(t)\right|\dt \geq \tfrac{j}{2},
$$
which is a contradiction. If $\left\{\tfrac{1}{2}(a_{j}+b_{j})\, \big|\, j\in\cN\right\}$ clusters only at $-\infty$, then with a similar method we conclude a contradiction.
\end{proof}

Now we are in the position to prove the main result of this subsection, which says that $\wor$ is isometrically rigid.

\begin{theorem}\label{thm:mainW1R}
Let $\varphi\colon\wor\to\wor$ be an isometry, that is, a bijection satisfying
\begin{equation*}
    \dwo(\varphi(\mu),\varphi(\nu)) = \dwo(\mu,\nu) \qquad (\mu,\nu\in\wor).
\end{equation*}
Then $\varphi = \psi_\#$ for some $\psi\in\isom(\R)$. Therefore, we also have 
$$
\isom(\wor) = \isom(\R).
$$
\end{theorem}

\begin{proof}
Since $\varphi$ is an isometry, for every $\mu,\nu,\eta\in\wor, \mu\neq\nu$ we have 
$$
	\eta\in\mmn \quad\iff\quad \varphi(\eta)\in M(\varphi(\mu),\varphi(\nu)),
$$
and hence also $\mathrm{diam}\left(\mmn\right)=\mathrm{diam}\left(M(\varphi(\mu),\varphi(\nu))\right)$. By the above claims this implies that $\varphi$ preserves adjacency in both directions, and thus $\varphi$ leaves $\Delta(\R)$ invariant. 
Since we have $\dwo(\delta_x,\delta_y)=|x-y| \; (x,y\in\R)$, we easily obtain an isometry $\psi\colon\R\to\R$ such that 
$$
	\varphi(\delta_x) = \delta_{\psi(x)} \qquad x\in\R.
$$
If $\psi$ is the identity map on $\R$, then by an argument similar to the one in \eqref{eq:step2helyett1}--\eqref{eq:step2helyett2} in Claim \ref{claim:identity} we can conclude that $\varphi$ is the identity on $\wor$. If $\psi$ is not the identity, then we can replace $\varphi$ with $\varphi \psi^{-1}_{\#}$, and in this case we obtain $\varphi=\psi_{\#}$.
\end{proof}

We finish this subsection with two short remarks. First, we would like to point out a somewhat surprising consequence of Theorems \ref{thm:main} and \ref{thm:mainW1R}.

\begin{corollary}
Even though we have $\Wo\zo\subset\wor$, not every isometry of $\Wo\zo$ can be extended into an isometry of $\wor$.
Moreover, no subgroup of $\isom\ler{\Wo(\R)}$ is isomorphic to $\isom\ler{\Wo\zo}$.
\end{corollary}

\begin{proof}
	For the flip operation $j$ we have $j(\Delta\zo) \not\subseteq \Delta\zo$, however, every isometry of $\wor$ leaves the set of all Dirac masses invariant.
	As for the second statement, we note that in $\isom\ler{\Wo(\R)}$ the product of any two different elements of order two is never an element of order two, in contrast with the Klein group.
\end{proof}

Second, our proof of the above theorem strongly relies on the assumption that $\varphi$ is bijective, since for a general $\varphi\in\isemb\ler{\Wo(\R)}$ we usually have $\varphi\ler{\mmn}\subsetneq M\ler{\varphi(\mu),\varphi(\nu)}$. In fact, one can easily construct non-surjective examples with essentially different properties. We continue with two such examples. The first one is the translation on the space of quantile functions. These maps will be crucial in the next section, so we start by a definition for $p\geq1$.

\begin{definition}[Translation in $\Wp(\R)$]\label{def:translation}
	Let $p\geq1$ and $\nu\in\Wp(\R)$ be arbitrary. Then the map defined by
	$$
		\varphi\colon\Wp(\R)\to\Wp(\R), \;\;\; F^{-1}_{\varphi\ler{\mu}} = F^{-1}_\mu + F^{-1}_\nu \;\;\; (\mu\in\Wp(\R))
	$$
	is called a \emph{translation} by the measure $\nu$. By \eqref{eq:wp-tav}, this defines an isometric embedding. Clearly, a translation is bijective if and only if $\nu\in\Delta(\R)$.
\end{definition}
{Before we continue, we note that translation can be interpreted as summing random variables with laws $\mu$ and $\nu$. While convolution corresponds to summing independent random variables, $F_{\mu}^{-1}+F_{\nu}^{-1}$ corresponds to summing random variables that are most closely coupled.}
Let us now point out that if $\varphi$ is the translation by $\tfrac{1}{2}\delta_{-1} + \tfrac{1}{2}\delta_{1}$, then the range of $\varphi$ contains only such measures whose support is never the whole $\R$, as the quantile function of each $\varphi(\mu)$ jumps at $\tfrac{1}{2}$. If $\varphi$ is the translation by the uniform measure on $[0,1]$, then all the slopes of each $F^{-1}_{\varphi(\mu)}$ are at least 1. 
Therefore all the slopes of each $F_{\varphi(\mu)}$ must be at most 1, hence $\varphi$ only contains absolutely continuous measures in its range. One could say that such a translation ``smoothens out'' measures.

Our second example is special in the sense that its range contains only measures which coincide on the open interval $(-1,1)$. Let $E\colon [-1,1) \to [\tfrac{1}{3},\tfrac{2}{3}]$ be an arbitrary right-continuous and monotone increasing function. 
We define $\varphi\colon\wor\to\cP(\R)$ by
\begin{equation*}
F_{\varphi(\mu)}(x) = 
\left\{\begin{matrix}
	\tfrac{1}{3}F_\mu(\tfrac{x+1}{3}) & \text{if}\; x < -1, \\
	E(x) & \text{if}\; -1\leq x <1,\\
	\tfrac{2}{3} + \tfrac{1}{3}F_\mu(\tfrac{x-1}{3}) & \text{if}\; 1\leq x.
\end{matrix}\right..
\end{equation*}
It is easy to see that indeed $\varphi$ maps $\wor$ into $\cP(\R)$.
Next, notice that the following holds for all $\mu,\nu\in\wor$:
\begin{align*}
	\int_{-\infty}^\infty|F_{\varphi(\mu)}(x) &- F_{\varphi(\nu)}(x)|~\mathrm{d}x\\
	&= \int_{-\infty}^{-1} \tfrac{1}{3}|F_\mu(\tfrac{x+1}{3}) - F_\mu(\tfrac{x+1}{3})|~\mathrm{d}x + \int_{1}^\infty \tfrac{1}{3}|F_\mu(\tfrac{x-1}{3}) - F_\mu(\tfrac{x-1}{3})|~\mathrm{d}x \\ 
	&= \int_{-\infty}^\infty |F_{\mu}(x) - F_{\nu}(x)|~\mathrm{d}x = \dwo(\mu,\nu).
\end{align*}
Therefore, substituting $\nu = \delta_0$ and noticing that $\varphi(\delta_0)$ is supported on $[-1,1]$ shows that $\varphi$ maps $\wor$ into itself. Hence it is an isometric embedding of $\wor$.

In contrast to the above examples, one may observe the following fact which shows at least some kind of a rigidity of the Wasserstein space $\wor$.

\begin{proposition}
	Let $\varphi$ be an isometric embedding of $\wor$ such that $\varphi(\Delta(\R))\subseteq \Delta(\R)$.
	Then $\varphi$ is an isometry of $\wor$.	
\end{proposition}

We omit the proof, as one can easily do it using the ideas of Theorem \ref{thm:mainW1R}.

We have seen that $\isom\ler{\wor)}$ and $\isom\ler{\mathcal{W}_2(\R))}$ are essentially different. To get the full picture, we continue by investigating $\isom\ler{\Wp(\R)}$ in the case of $p>1$, $p\neq2$.

\subsection{$p>1,\,p\neq2$ -- Characterization of isometric embeddings}\label{s:wpr}
Similarly to Section \ref{s:II}, here we are able to handle the more general case of isometric embeddings. This time, however, $\isom(\Wp(\R))$ and $\isemb(\Wp(\R))$ are different. We will show that isometric embeddings are compositions of trivial isometries and translations (see Definition \ref{def:translation}). In particular, it will turn out that every isometric embedding that leaves the set of all Dirac masses $\Delta(\R)$ invariant is a trivial isometry. In this subsection it is more convenient to consider $\Wp(\R)$ as a space of quantile functions, hence, as a subset of $L^p\big((0,1)\big)$.
In order to achieve our goal first, we prove an abstract Mankiewicz-type lemma which ensures that every isometric embedding of $\Wp(\R)$ can be extended to an isometric embedding of $L^p((0,1))$. 
Then we apply the Banach--Lamperti theorem which describes all linear isometric embeddings of $L^p((0,1))$ for $p > 1$, $p\neq 2$, see \cite[Theorem 3.1]{Lamperti}.

We call a convex subset of a real Banach space with non-empty interior a convex body.
Also, when we talk about bijective distance preserving maps between \emph{two different} metric spaces, then we will call them simply isometries.
When bijectivity is not assumed, then we shall call them isometric embeddings.

Now, we state Mankiewicz's theorem.

\begin{theorem}[Mankiewicz, \cite{Mankiewicz}]
	Let $X$ and $Y$ be two real Banach spaces and $K\subset X$, $M\subset Y$ be convex bodies.
	Then every isometry $\phi\colon K\to M$ can be extended to an (affine) isometry $\Phi\colon X\to Y$.
\end{theorem}

We note that every isometric embedding of a Banach space $X$ into a strictly convex Banach space $Y$ is automatically affine (linear up to translation).
Indeed, in this case the strict triangle inequality holds in $Y$, hence the midpoint $y$ of any two points $y_1,y_2\in Y$ is characterized by 
$$
\|y-y_1\| = \|y-y_2\| = \tfrac{1}{2} \|y_1-y_2\|.
$$
Also, the Mazur--Ulam theorem ensures that all isometries between two Banach spaces are affine, although note that this statement fails for isometric embeddings in general.

The reason why we cannot apply Mankiewicz's theorem directly is that although $\Wp(\R)$ is a convex and closed subset of $L^p((0,1))$, its interior is empty.
However, since $L^p((0,1))$ is a strictly convex Banach space, we can overcome this obstacle with the forthcoming lemma. 
The linear span of a set $K\subseteq X$ will be denoted by $\linspan(K)$, and its closure by $(\linspan(K))^-$.

\begin{lemma}\label{lem:extnon2}
	Let $X$ be a real, strictly convex Banach space and $K\subset X$ be a convex set (with possibly empty interior) such that $0\in K$ and $(\linspan(K))^- = X$.
	Then every isometric embedding $\varphi\colon K\to X$ with $\varphi(0) = 0$ can be uniquely extended to a (linear) isometric embedding $L\colon X\to X$.
\end{lemma}

\begin{proof}
	We only need to extend $\varphi$ to the dense subspace $\linspan K$, as from there extending to the whole space is straightforward by a simple continuity and completeness argument.
	Let us define a set of finite dimensional subspaces, where $\Lat(X)$ denotes the lattice of all subspaces of $X$:
	$$
		\S := \{ M\in\Lat(X) \; |\; \dim M < \infty, M\cap K \text{ is a convex body in } M \}.
	$$
	Now, we prove that
	\bes
		\linspan K = \cup \{ M \; |\; M\in\S \}.
	\ees
	Indeed if $x = \sum_{j=1}^m a_j x_j, \; a_j\in\R, x_j\in K$, then 
	$$
	x \in M := \linspan\{ x_1,\dots x_m \} = \linspan\{ x_{i_1},\dots x_{i_k} \}
	$$ 
	where $1\leq i_1 < i_2 < \dots < i_k \leq m$ and the system $x_{i_1},\dots x_{i_k} \in K$ is a base in $M$.
	Obviously, the simplex spanned by 0 and this system is a convex body in $M$, hence so is $K\cap M$.
	Therefore we obtain $\linspan K \subseteq \cup \{ M \; | \; M\in\S \}$.
	For the reverse, let $M \in \S$ with $\dim M = m$, then by definition there must exist $m+1$ affine independent points $x_0, x_1, \dots, x_m$ in $K\cap M$.
	Clearly, then $\{x_i-x_0\}_{i=1}^m$ is a base in $M$, and hence $M \subseteq \linspan K$.

	Next, we show that for every $M\in\S$ there exists a unique linear extension of $\varphi|_{K\cap M}$ to $M$, which also happens to be an isometric embedding.
	By strict convexity, $\varphi$ is an affine map, which also fixes 0. 
	Therefore the restriction $\varphi|_{K\cap M}$ can be extended to a unique injective linear map 
	$$
		L_M\colon M\to \linspan(\varphi(K\cap M)).
	$$
	Clearly, $\varphi(K\cap M)$ is a convex body in $\linspan(\varphi(K\cap M))$, thus by Man\-kie\-wicz's theorem $L_M$ must be an isometry too.
	Also, note that by construction we have
	\begin{equation}\label{eq:biggerS}
		M,N\in\S, M \subseteq N \;\Longrightarrow\; L_N|_M = L_M.
	\end{equation}
	
	Now, we have an extension for every $M\in\S$, and our goal is to show that if $M,N\in \S$, then $L_M|_{M\cap N} = L_N|_{M\cap N}$.
	(However, we point out that $M,N\in \S$ does not imply $M\cap N\in \S$ in general.)
	This will show that the following map is a well-defined, linear, distance-preserving extension of $\varphi$:
	$$
		L\colon \linspan K \to X, \; Lx = L_M x \; \text{where } x\in M\in\S,
	$$
	and thus the proof will be complete.
	For this observe that
	\bes
		M, N \in \S \; \Longrightarrow \; M+ N \in \S.
	\ees
	Indeed, assume indirectly that $M+N \notin \S$, thus $(M+N)\cap K$ is not a convex body in $M+N$.
	This also means that $(M+N)\cap K$ spans an affine subspace $E$ of $M+N$ with co-dimension at least 1.
	As $0\in K$, the affine subspace $E$ is a linear subspace.
	But, as both $M\cap K$ and $N\cap K \subset (M+N)\cap K \subset E$ and they are convex bodies in $M$ and $N$, respectively, we get that $M, N\subset E$ and hence $M+N \subseteq E$, a contradiction.
	Therefore by \eqref{eq:biggerS}, for every $M,N\in \S$ we obtain 
	$$
		L_M|_{M\cap N} = L_{M+N}|_{M\cap N} = L_N|_{M\cap N},
	$$
	which completes the proof.
\end{proof}

As we mentioned earlier, $\Wp(\R)$ is a convex and closed subset of $L^p((0,1))$, where we regard elements of $\Wp(\R)$ as quantile functions.
Let us point out that $\linspan \ler{\Wp(\R)}$ is dense in $L^p((0,1))$, since the functions $t\mapsto t^n$, $n\in\N$, are elements of $\Wp(\R)$.
Therefore, by the above lemma, if $\varphi\in \mathrm{IsEmb}\ler{\Wp(\R)}$ and $\varphi$ fixes $\delta_0$, then $\varphi$ can be extended to a linear isometric embedding of $L^p((0,1))$.
The latter have been characterized in \cite[Theorem 3.1]{Lamperti}, which we recall now.

\begin{definition}[Regular set-isomorphism]
Let $((0,1),\mathcal{L}_{(0,1)},\lambda)$ be the measure space where $\mathcal{L}_{(0,1)}$ stands for the $\sigma$-algebra of all Lebesgue sets of $(0,1)$ and $\lambda$ is the normalized Lebesgue measure.
We call a map $T\colon \mathcal{L}_{(0,1)} \to \mathcal{L}_{(0,1)}$, defined modulo sets of measure zero, a \emph{regular set-isomorphism} if the following conditions hold:
\begin{itemize}
	\item[(a)] $T((0,1)\setminus A) = T((0,1))\setminus T(A)$ for all Lebesgue sets $A\subseteq (0,1)$,
	\item[(b)] $T\left( \cup_{n=1}^\infty A_n \right) = \cup_{n=1}^\infty T(A_n)$ for disjoint Lebesgue sets $A_n\subseteq (0,1)$,
	\item[(c)] for each Lebesgue set $A\subseteq (0,1)$, we have $\lambda(T(A)) = 0$ if and only if $\lambda(A) = 0$.
\end{itemize}
A regular set-isomorphism induces a linear transformation on the set of all Lebesgue-measurable functions, which is also denoted by $T$, and which is characterized by $T\bigchi_A = \bigchi_{T(A)}$ where $\bigchi_A$ denotes the characteristic function of a Lebesgue set $A$ (see \cite{Lamperti} for more details).
\end{definition}

We note that a regular set-isomorphism does not need to be bijective, for instance the map $T(A) := \frac{1}{2} A = \left\{\frac{1}{2} x \; | \; x\in A\right\}$ defines a non-bijective one.
In particular, as can be seen from the Banach--Lamperti theorem below, a typical linear isometric embedding of $L^p((0,1))$ is in fact not bijective. From now on, we will denote the constant function with value $1$ by $\mathbf{1}$. 

\begin{theorem}[Banach--Lamperti]
	Let $1\leq p <\infty$, $p\neq 2$ be a fixed parameter, and assume that $U\colon L^p((0,1)) \to L^p((0,1))$ is a linear isometric embedding.
	Then there exists a regular set-isomorphism $T$ of the measure space $((0,1),\mathcal{L}_{(0,1)},\lambda)$ such that
	\bes
		(Uf)(x) = h(x)\cdot (Tf)(x) \qquad (\text{a.e.~}x\in (0,1)),
	\ees
	where $h = U\bigchi_{(0,1)} = U {\bf 1} = U F^{-1}_{\delta_{1}} \in L^p((0,1))$.
\end{theorem}

In \cite{Lamperti} the reader can find a more general statement that holds for all $\sigma$-finite measure spaces, and which includes a converse statement too. 
However, we shall only need the above very special version.
Note that even though $h = U {\bf 1}$ is not explicitly stated in \cite[Theorem 3.1]{Lamperti}, it can be found in its proof in case when the measure space is finite.
Before we can apply the Banach--Lamperti theorem, we have to show that the study of a general isometric embeddings of $\Wp(\R)$, $1<p<\infty$, can be reduced to the study of those isometric embeddings of $\Wp(\R)$ that fix $\delta_0$.
This is what we do in the next lemma.

\begin{lemma}\label{lem:Dirac2Dirac}
	Assume that $1<p<\infty$ and $\varphi \colon \Wp(\R)\to \Wp(\R)$ is an isometric embedding, that is,
	$$
		\left\|F^{-1}_{\varphi(\mu)} - F^{-1}_{\varphi(\nu)}\right\|_p = \left\|F^{-1}_{\mu} - F^{-1}_{\nu}\right\|_p \quad (\mu,\nu\in\Wp(\R)).
	$$
	Then either
	$$
		F^{-1}_{\varphi(\delta_t)} = F^{-1}_{\varphi(\delta_0)} + t\cdot{\bf 1} \qquad (t\in\R),
	$$
	or
	$$
		F^{-1}_{\varphi(\delta_t)} = F^{-1}_{\varphi(\delta_0)} - t\cdot{\bf 1} \qquad (t\in\R).
	$$
	Moreover, we have
	\begin{equation}\label{eq:phitildeertelmes}
		F^{-1}_{\varphi(\mu)} - F^{-1}_{\varphi(\delta_0)} \in \Wp(\R) \qquad (\mu\in\Wp(\R)).
	\end{equation}
	In particular, the mapping $\tilde\varphi\colon \Wp(\R)\to \Wp(\R)$ defined by
	\begin{equation}\label{eq:ujphi}
		F^{-1}_{\tilde\varphi(\mu)} := F^{-1}_{\varphi(\mu)} - F^{-1}_{\varphi(\delta_0)} \qquad (\mu\in\Wp(\R))
	\end{equation}
	is a well-defined isometric embedding such that either $\tilde\varphi(\delta_t) = \delta_t$ for all $t\in\R$, or $\tilde\varphi(\delta_t) = \delta_{-t}$ for all $t\in\R$.
\end{lemma}

\begin{proof}
	First, we show that $F^{-1}_{\varphi(\delta_1)} - F^{-1}_{\varphi(\delta_0)}\in\{-\mathbf{1},\mathbf{1}\}.$
	By strict convexity of the norm, we have $F^{-1}_{\varphi(\delta_x)}(t) = (1-x)\cdot F^{-1}_{\varphi(\delta_0)}(t) + x\cdot F^{-1}_{\varphi(\delta_1)}(t)$ for all $x\in\R$ and $0<t<1$.
	Thus for all $x\in\R$ and $0<t_1<t_2<1$ we obtain
	\begin{align*}
		0&\leq \frac{F^{-1}_{\varphi(\delta_x)}(t_1) - F^{-1}_{\varphi(\delta_x)}(t_2)}{t_1-t_2}= (1-x)\cdot \frac{F^{-1}_{\varphi(\delta_0)}(t_1) - F^{-1}_{\varphi(\delta_0)}(t_2)}{t_1-t_2} + x\cdot \frac{F^{-1}_{\varphi(\delta_1)}(t_1) - F^{-1}_{\varphi(\delta_1)}(t_2)}{t_1-t_2}.
	\end{align*}
	Notice that this happens if and only if all the slopes $\tfrac{F^{-1}_{\varphi(\delta_0)}(t_1) - F^{-1}_{\varphi(\delta_0)}(t_2)}{t_1-t_2}$ and $\tfrac{F^{-1}_{\varphi(\delta_1)}(t_1) - F^{-1}_{\varphi(\delta_1)}(t_2)}{t_1-t_2}$ coincide, or in other words, $F^{-1}_{\varphi(\delta_1)} - F^{-1}_{\varphi(\delta_0)}$ is constant on $(0,1)$.
	Taking into account the distances, $F^{-1}_{\varphi(\delta_1)} - F^{-1}_{\varphi(\delta_0)} \in \left\{{\bf 1}, -{\bf 1}\right\}$ follows.
	By strict convexity we then easily get the first statement of the lemma.
	
	Next, let $\mu,\nu\in\Wp(\R)$ be arbitrary, and let us define the set:
	\begin{align*}
		I_{\mu,\nu} :&= \left\{ s\in\R \; |\; (1-s)\cdot F^{-1}_\mu + s\cdot F^{-1}_\nu \text{ is monotone increasing} \right\} \\
		& = \left\{ s\in\R \; |\; \exists\; \eta \in\Wp(\R) \colon (1-s)\cdot F^{-1}_\mu + s\cdot F^{-1}_\nu = F^{-1}_\eta \right\}.
	\end{align*}
	Clearly, the set $I_{\mu,\nu}$ is always a closed interval that contains $[0,1]$.
	Now, we observe that for any $\mu,\nu\in\Wp(\R)$ we have $[0,\infty) \subseteq I_{\mu,\nu}$ if and only if
	\begin{equation}\label{eq:0vegtelen}
		\frac{F^{-1}_{\mu}(t_1) - F^{-1}_{\mu}(t_2)}{t_1-t_2} \leq \frac{F^{-1}_{\nu}(t_1) - F^{-1}_{\nu}(t_2)}{t_1-t_2} \qquad (0<t_1<t_2<1).
	\end{equation}
	In particular, we always have $[0,\infty) \subseteq I_{\delta_0,\mu}$, thus also $[0,\infty) \subseteq I_{\varphi(\delta_0),\varphi(\mu)}$ for all $\mu\in\Wp(\R)$.
	Therefore applying \eqref{eq:0vegtelen} with $\varphi(\delta_0)$ and $\varphi(\mu)$, we obtain \eqref{eq:phitildeertelmes} and the rest of the statement follows easily.
\end{proof}

Now, we are in the position to prove our theorem on $\mathrm{IsEmb}(\Wp(\R))$ for $p>1, p\neq 2$, using the Banach--Lamperti theorem and our Mankiewicz-type extension lemma.

\begin{theorem}\label{thm:mainWpR}
Let $1<p<\infty, p\neq 2$ and $\varphi\in\isemb(\Wp(\R))$, that is
\bes
\dwp\ler{\varphi(\mu),\varphi(\nu)} = \dwp\ler{\mu, \nu} \qquad (\mu, \nu \in \Wp(\R)).
\ees
Then $\varphi$ is a composition of a trivial isometry and a translation, that is, there exists a $\psi\in\mathrm{Isom}(\R)$ and $\nu\in\Wp(\R)$ such that
\begin{equation}\label{eq:WpIsEmbform}
	F^{-1}_{\varphi(\mu)} = F^{-1}_{\psi_\#(\mu)} + F^{-1}_{\nu} \qquad (\mu\in\Wp(\R)).
\end{equation}
In particular, if $\varphi$ is also bijective, then it is a trivial isometry, therefore we have 
$$
	\mathrm{Isom}(\R) = \mathrm{Isom}(\Wp(\R)) \subsetneq \isemb(\Wp(\R)).
$$
\end{theorem} 

\begin{proof}
Consider the mapping $\tilde\varphi\in\isemb(\Wp(\R))$ defined in \eqref{eq:ujphi}, which either fixes all Dirac measures, or $\tilde\varphi(\delta_x) = \delta_{-x}$ for all $x\in\R$.
Therefore, it is enough to show that if an isometric embedding $\tilde\varphi$ fixes all Dirac measures, then it fixes every measure in $\Wp(\R)$.

By Lemma \ref{lem:extnon2}, we conclude that the map $\tilde\varphi$ is a restriction of a linear isometric embedding $U\colon L^p((0,1))\to L^p((0,1))$, which fixes all constant functions.
Hence, in the Banach--Lamperti theorem we have $h(x) = 1$ for a.e.~$x$, therefore
\begin{equation}\label{eq:UT}
	Uf(x) = Tf(x) \qquad (\text{a.e.~}x\in (0,1))
\end{equation}
where $T$ is the linear operator generated by a regular set-isomorphism $T$ (which is defined modulo null-sets).
Substituting 
$$
f = \bigchi_{[a,1)} = F^{-1}_{a\delta_0+(1-a)\delta_1}\quad \text{and} \quad f = -\bigchi_{(0,a)} = F^{-1}_{a\delta_{-1}+(1-a)\delta_0}
$$
into \eqref{eq:UT} gives
$$
	\bigchi_{T([a,1))}, -\bigchi_{T((0,a))} \in\Wp(\R)\subset L^p\ler{(0,1)} \qquad (0<a<1).
$$
Therefore, by the properties of the regular set-isomorphism $T$, for every $0<a<1$ there exists a $0<t_a< 1$ such that
$T([a,1)) = [t_a,1)$ and $T((0,a)) = (0,t_a)$.
Hence we obtain
$$
	U \bigchi_{[a,1)} = 	\bigchi_{[t_a,1)} \;\; \text{and} \;\; U \bigchi_{(0,a)} = \bigchi_{(0,t_a)} \qquad (0<a< 1),
$$
but since $U$ preserves the $p$-norm, the number $t_a$ must coincide with $a$ for all $0<a<1$.
Thus we get that $U$ is the identity operator, and the proof of \eqref{eq:WpIsEmbform} is complete.

Finally, note that if $\varphi$ is also assumed to be bijective, then $F^{-1}_{\nu}$ must be a constant function.
Indeed, otherwise there would exist two points $0<t_1<t_2<1$ such that 
$$
\frac{F^{-1}_{\varphi(\mu)}(t_1)-F^{-1}_{\varphi(\mu)}(t_2)}{t_1-t_2}\geq \frac{F^{-1}_{\nu}(t_1)-F^{-1}_{\nu}(t_2)}{t_1-t_2} > 0 \quad \ler{\mu\in\Wp(\R)},
$$ 
and therefore Dirac masses would not be in the range of $\varphi$.
\end{proof}

\subsection{$p=2$ -- A functional analytic description of the exotic flow}
\label{s:w2r}
The aim of this subsection is to take a closer look at Kloeckner's surprising result on $\isom\ler{\mathcal{W}_2(\R)}$.
We introduce the notation $m(\mu)$ for the center of mass of a $\mu\in\mathcal{W}_1(\R)$:
$$
m(\mu) = \int_0^1 F^{-1}_\mu(x)~\mathrm{d}x.
$$
The map $r_c\colon\R\to\R, x\mapsto 2c-x$ is called the reflection through $c\in\R$.

Kloeckner showed in \cite[Theorem 1.1]{Kloeckner-2010} that the group $\isom(\mathcal{W}_2(\R))$ is the semidirect product $\isom(\R)\ltimes\isom(\R)$. Namely, he showed that every isometry of $\mathcal{W}_2(\R)$ is a composition of some of the of the following maps:
\begin{itemize}
	\item[(1)] a trivial isometry, that is, $\psi_\#$ for some $\psi\in\isom(\R)$;
	\item[(2)] the map $\mu\mapsto{(r_{m(\mu)})}_\#(\mu)$, that is, the isometry that reflects every measure through its center of mass; and
	\item[(3)] a so-called exotic isometry $\Phi^q$ for some $q\in\R$, which we focus on in this subsection, see the definition in \eqref{eq:Phiqdef1} and \eqref{eq:Phiqdef2}.
\end{itemize}
Note that all of the above types of isometries leave the set of all Dirac measures invariant, moreover, (2)--(3) fix every Dirac measures.
Also, in cases (1)--(2) the support of each measure $\mu$ is isometrically congruent to the support of its image.
However, this is not the case for exotic isometries.

The set of all measures which are supported on at most two points will be denoted by $\Delta_2(\R)$.
As in \cite{Kloeckner-2010}, we parametrize $\Delta_2(\R)$ by $x,p\in\R$, $\sigma\geq 0$ as
\begin{equation}\label{eq:Phiqdef1}
	\mu(x,\sigma,p) := \frac{e^{-p}}{e^{p}+e^{-p}}\cdot \delta_{x-\sigma e^p} + \frac{e^{p}}{e^{p}+e^{-p}}\cdot \delta_{x+\sigma e^{-p}}.
\end{equation}
Let $q\in\R$ be fixed. Using the above parametrization, Kloeckner defined the exotic isometry $\Phi^q$ on $\Delta_2(\R)$ in the following way:
\begin{equation}\label{eq:Phiqdef2}
	\Phi^q\left(\mu(x,\sigma,p)\right) := \mu(x,\sigma,p+q)  \quad (x,\sigma,p\in\R, \sigma\geq 0).
\end{equation}
He proved that this indeed defines an isometry on $\Delta_2(\R)$ and that it extends uniquely to an isometry of $\mathcal{W}_2(\R)$.
He also pointed out that even though the above definition is constructive, it is not very explicit outside $\Delta_2(\R)$.
Moreover, he illustrated that an explicit formula for general measures supported on three points already seems to be very complicated.

The goal of this subsection is to provide a general explicit formula for the action of exotic isometries, which we shall do by using functional analytic techniques rather than geometric ones. Namely, we prove that if we regard elements of $\mathcal{W}_2(\R)$ as quantile functions, then $\Phi^q$ extends to a real unitary operator $U_ q:L^2((0,1))\to L^2((0,1))$ which can be written in terms of a composition operator, the Volterra operator, a multiplication operator, and a rank-one projection. First we state a well-known lemma (see for instance \cite[Theorem 11.4]{WW}) that will be helpful in our considerations.

\begin{lemma}\label{lem:extensionWt}
	Let $H$ be a real Hilbert space and $S$ be a subset such that $0\in S$ and $\linspan{S}$ is dense in $H$.
	If $\varphi\colon S\to H$ is an isometric embedding such that $\varphi(0) = 0$, then it can be uniquely extended to a (linear) isometric embedding $L\colon H\to H$.
\end{lemma}

Now our theorem introduced above reads as follows.

\begin{theorem}\label{thm:mainW2R}
Let $q$ be a real number. Then the action of the exotic isometry $\Phi^q$ is given by the following formula:
\begin{align}\label{eq:exoticform}
	F^{-1}_{\Phi^q(\mu)} (x) = & \;(1-e^q)\cdot m(\mu) + \left\{e^q+(e^{-q}-e^q)h_q(x)\right\}\cdot F^{-1}_\mu(h_q(x)) \\
	&+ (e^{q}-e^{-q})\cdot\int_0^{h_q(x)} F^{-1}_\mu(s)~\mathrm{d}s \qquad (\mu\in\Wt(\R), 0<x<1),\nonumber
\end{align}
where
\begin{equation}\label{eq:hq}
	h_q(x) = \frac{x e^{2q}}{1+(e^{2q}-1)x} \qquad (x\in(0,1)).
\end{equation}

\end{theorem} 

\begin{proof}
	{Again, we will regard the Wasserstein space $\mathcal{W}_2(\R)$ as a convex and closed subset of $L^2((0,1))$ whose linear span is dense in $L^2((0,1))$.} Therefore by Lemma \ref{lem:extensionWt} and \eqref{eq:Phiqdef2} the exotic isometry $\Phi^q$ can be extended to a unique linear isometric embedding which we denote by $U_q$. 
	Let us point out that since $\Delta_2(\R)$ is such a subset of $L^2((0,1))$ whose linear span is dense, therefore $U_q$ is the unique bounded linear operator on $L^2((0,1))$ which satisfies
	\begin{equation}\label{eq:Uq}
		U_q\left( F^{-1}_{\mu(x,\sigma,p)} \right) = F^{-1}_{\mu(x,\sigma,p+q)} \qquad (x,\sigma,p\in\R, \sigma\geq 0).
	\end{equation}
	Therefore, it is enough to find a bounded linear operator (without proving its isometric property) that satisfies \eqref{eq:Uq}.
	Also, observe that \eqref{eq:Uq} is equivalent to $U_q {\bf 1} = {\bf 1}$, and $U_q\left( F^{-1}_{\mu_p} \right) = F^{-1}_{\mu_{p+q}}$ $(p\in\R)$, where we used the shorthand $\mu_p := \mu(0,1,p)$. Define the operator
	$$
		T_q\colon L^2((0,1))\to L^2((0,1)), \;\;\; T_q = U_q - {\bf 1}\otimes {\bf 1},
	$$
	where $({\bf 1}\otimes {\bf 1}) f = \int_0^1 f(s)~\mathrm{d}s \cdot {\bf 1}$ is the rank-one projection onto the subspace of all constant functions.
	Next, we describe how $T_q$ acts on certain characteristic functions.
	Since for all $p\in\R$
	$$
		F^{-1}_{\mu_p} = -e^p\cdot \bigchi_{\left(0,\tfrac{e^{-p}}{e^{-p}+e^p}\right)} + e^{-p}\cdot\bigchi_{\left[\tfrac{e^{-p}}{e^{-p}+e^p},1\right)}
	$$
	holds, we calculate
	\begin{align*}
		T_q \bigchi_{\left(0,\tfrac{e^{-p}}{e^{-p}+e^p}\right)} 
		& = \tfrac{1}{e^{-p}+e^p}\cdot U_q \left( e^{-p}\cdot{\bf 1} - F^{-1}_{\mu_p} \right) - \tfrac{e^{-p}}{e^{-p}+e^p}\cdot{\bf 1}
		= \tfrac{-1}{e^{-p}+e^p}\cdot F^{-1}_{\mu_{p+q}} \\
		& = \tfrac{e^{p+q}}{e^{-p}+e^p}\cdot \bigchi_{\left(0,\tfrac{e^{-p-q}}{e^{-p-q}+e^{p+q}}\right)} - \tfrac{e^{-p-q}}{e^{-p}+e^p}\cdot\bigchi_{\left[\tfrac{e^{-p-q}}{e^{-p-q}+e^{p+q}},1\right)}.
	\end{align*}
	Notice that since $T_q {\bf 1} = 0\cdot {\bf 1}$, we have
	\begin{align*}
		T_q \bigchi_{\left[\tfrac{e^{-p}}{e^{-p}+e^p},1\right)} = -\tfrac{e^{p+q}}{e^{-p}+e^p}\cdot \bigchi_{\left(0,\tfrac{e^{-p-q}}{e^{-p-q}+e^{p+q}}\right)} + \tfrac{e^{-p-q}}{e^{-p}+e^p}\cdot\bigchi_{\left[\tfrac{e^{-p-q}}{e^{-p-q}+e^{p+q}},1\right)}.
	\end{align*}
	Now, we define a set and two transformations on it.
	Let
	$$
		\cI := \left\{ \bigchi_{(0,t)} \,\big|\, 0<t<1 \right\} \cup \left\{ \bigchi_{[t,1)} \,\big|\, 0<t<1 \right\},
	$$
	\begin{align*}
		T_q^{(0)}\colon \cI\to L^2((0,1)), \;\;\;
		& T_q^{(0)} \left( \bigchi_{\left(0,\tfrac{e^{-p}}{e^{-p}+e^p}\right)}  \right) = \tfrac{e^{p}}{e^{-p}+e^p}\cdot \bigchi_{\left(0,\tfrac{e^{-p-q}}{e^{-p-q}+e^{p+q}}\right)} \\
		& T_q^{(0)} \left( \bigchi_{\left[t,1\right)} \right) = -T_q^{(0)} \left( \bigchi_{\left(0,t\right)} \right) \quad (p\in\R, 0<t<1)
	\end{align*}
	and
	\begin{align*}
		T_q^{(1)}\colon \cI\to L^2((0,1)), \;\;\;
		& T_q^{(1)} \left( \bigchi_{\left[\tfrac{e^{-p}}{e^{-p}+e^p}, 1\right)}  \right) 
		= \tfrac{e^{-p}}{e^{-p}+e^p}\cdot\bigchi_{\left[\tfrac{e^{-p-q}}{e^{-p-q}+e^{p+q}},1\right)} \\
		& T_q^{(1)} \left( \bigchi_{\left(0,t\right)} \right) = -T_q^{(1)} \left( \bigchi_{\left[t,1\right)} \right) \quad (p\in\R, 0<t<1).
	\end{align*}
	Observe that $T_q|_{\cI} = e^q \cdot T_q^{(0)} + e^{-q} \cdot T_q^{(1)}$.
	However, at this point we do not know if $T_q^{(0)}$ or $T_q^{(1)}$ can be extended linearly and continuously to the whole space, thus we cannot treat them as operators.
	A calculation gives
	$$
		T_q^{(0)} \left( \bigchi_{\left(0,t\right)} \right) = C_q \circ T_0^{(0)} \left( \bigchi_{\left(0,t\right)} \right) 
		\;\;\;\text{and}\;\;\;
		T_q^{(0)} \left( \bigchi_{\left[t,1\right)} \right) = C_q \circ T_0^{(0)} \left( \bigchi_{\left[t,1\right)} \right)
	$$
	for all $0<t<1$ where
	$$
		C_q\colon L^2((0,1))\to L^2((0,1)), \;\;\; (C_q f)(x) = f(h_q(x)) \quad (x\in (0,1))
	$$
	is a composition operator with symbol $h_q$, see \eqref{eq:hq}. 
	Notice that $C_q$ is a bounded operator, as $h_q$ maps $[0,1]$ bijectively onto itself, it is a smooth function on a neighbourhood of $[0,1]$, and its derivative is bounded from below by $e^{-2|q|}$ on $[0,1]$.
	
	Next, let $M_{{\bf 1}-{\bf x}}$ stand for the multiplication operator by the function ${\bf 1}-{\bf x}$ where ${\bf x}(t) = t$, and $V$ for the Volterra operator: $(Vf)(t) = \int_0^t f(s)~\mathrm{d}s$ $(t\in (0,1))$.
	We notice that
	$$
		T_0^{(0)} \left( \bigchi_{\left(0,t\right)} \right) = (1-t)\cdot \bigchi_{\left(0,t\right)} = (M_{{\bf 1}-{\bf x}} - {\bf 1}\otimes {\bf 1} + V) \bigchi_{\left(0,t\right)}
		\qquad (0<t<1).
	$$
	Furthermore, since 
	$$
	(M_{{\bf 1}-{\bf x}}-{\bf 1}\otimes {\bf 1} + V) \left( \bigchi_{\left(0,t\right)} + \bigchi_{\left[t,1\right)} \right) = (M_{{\bf 1}-{\bf x}}-{\bf 1}\otimes {\bf 1} + V) {\bf 1} = 0\cdot {\bf 1},
	$$ 
	we also have $T_0^{(0)} \left( \bigchi_{\left[t,1\right)} \right) = (M_{{\bf 1}-{\bf x}}-{\bf 1}\otimes {\bf 1} + V) \bigchi_{\left[t,1\right)}$ for all $0<t<1$.
	Therefore, we obtain that 
	$$
	T_0^{(0)} = (M_{{\bf 1}-{\bf x}}-{\bf 1}\otimes {\bf 1} + V)|_\cI.
	$$ 
	
	As for $T_q^{(1)}$, we calculate and notice the following for all $0<t<1$:
	$$
		T_q^{(1)} \left( \bigchi_{\left(0,t\right)} \right) = C_q \circ T_0^{(1)} \left( \bigchi_{\left(0,t\right)} \right),
		\;\;\;
		T_q^{(1)} \left( \bigchi_{\left[t,1\right)} \right) = C_q \circ T_0^{(1)} \left( \bigchi_{\left[t,1\right)} \right)
	$$
	and
	\begin{align*}
		T_0^{(1)} \left( \bigchi_{\left(0,t\right)} \right) &= -t \cdot \bigchi_{\left[t,1\right)} = -t\cdot {\bf 1} + t \cdot \bigchi_{\left(0,t\right)} \\
		&= \left( -{\bf 1}\otimes {\bf 1} + I - T_0^{(0)} \right) \bigchi_{\left(0,t\right)}
		= \left( M_{{\bf x}} - V \right) \bigchi_{\left(0,t\right)}.
	\end{align*}
	As $\left( M_{{\bf x}} - V \right) {\bf 1} = 0\cdot {\bf 1}$, we conclude that 
	$$
	T_0^{(1)} = \left( M_{{\bf x}} - V \right)|_\cI.
	$$
	
	Therefore, the above observations together imply
	\begin{align*}
		U_q|_\cI & = (T_q + {\bf 1}\otimes {\bf 1})|_\cI = e^q \cdot T_q^{(0)} + e^{-q} \cdot T_q^{(1)} + ({\bf 1}\otimes {\bf 1})|_\cI \\
		& = e^q \cdot C_q\circ T_0^{(0)} + e^{-q} \cdot C_q\circ T_0^{(1)} + ({\bf 1}\otimes {\bf 1})|_\cI \\
		& = \left( e^q \cdot C_q\cdot \left(M_{{\bf 1}-{\bf x}}-{\bf 1}\otimes {\bf 1} + V\right) + e^{-q} \cdot C_q \cdot \left(M_{{\bf x}} - V\right) + {\bf 1}\otimes {\bf 1}\right)\big|_\cI.
	\end{align*}
	Thus we eventually conclude 
	\begin{equation*}\label{eq:exoticoperatorform}
		U_q = C_q\cdot \left[
			(1-e^q)\cdot ({\bf 1}\otimes {\bf 1}) +
			e^q \cdot I +
			(e^{-q}-e^q)\cdot M_{\bf x}+
			(e^q - e^{-q}) \cdot V
		\right],
	\end{equation*}
	which implies \eqref{eq:exoticform}, at least for almost every $x\in(0,1)$.
	However, since two right-continuous functions are equal almost everywhere on $(0,1)$ if and only if they coincide on $(0,1)$, we easily conclude \eqref{eq:exoticform}.
\end{proof}

\section*{Acknowledgements}
{This paper is based on discussions made during research visits at the Institute of Science and Technology (IST) Austria, Klosterneuburg. We are grateful to the Erd\H{o}s group for the warm hospitality. We thank the anonymous referee for their valuable comments to the manuscript and their helpful suggestions for changes. We are also grateful to Lajos Moln\'ar for his comments on earlier versions of the manuscript, and to L\'aszl\'o Erd\H{o}s for his suggestions on the structure and highlights of this paper.}


\begin{thebibliography}{99}
\bibitem{AGS}
L. Ambrosio, N. Gigli, and G. Savare, Gradient Flows
In Metric Spaces and in the Space of Probability Measures,
Lectures in Mathematics, ETH Zürich, Birkhäuser Verlag, 2005.

\bibitem{m1}
M. Arjovsky, S. Chintala, L. Bottou, \emph{Wasserstein Generative Adversarial Networks}, Proccedings of Machine Learning Research (2017), 214--223.

\bibitem{bertrand-kloeckner-hadamard}
J. Bertrand, and B. Kloeckner, \emph{A geometric study of Wasserstein spaces: Hadamard spaces} J. Topol. Anal. { 4}(4) (2012), 515--542.


\bibitem{bertrand-kloeckner-2016}
J. Bertrand, and B. Kloeckner, \emph{A geometric study of Wasserstein spaces: isometric rigidity in negative curvature,} Int. Math. Res. Notices {2016 (5)}, 1368--1386.


\bibitem{bgl}
E. Boissard, T. Le Gouic, J.-M. Loubes,
\emph{Distribution’s template estimate with Wasserstein}, Bernoulli {21}(2) 2015, 740--759.


\bibitem{Butkovsky}
O. Butkovsky, \emph{Subgeometric rates of convergence of Markov processes in the Wasserstein metric,} Annals of Appl. Prob {24}(2) (2014), 526--552.



\bibitem{dolinar-molnar} G. Dolinar, and L. Moln\'ar, \emph{Isometries of the space of distribution functions with respect to the Kolmogorov--Smirnov metric,} J. Math. Anal. Appl. {\bf 348} (2008), 494--498.

%


\bibitem{geher-kuiper} Gy. P. Geh\'er, \emph{Surjective Kuiper isometries},  Houston J. Math. 44 (2018), 263--281.


\bibitem{geher-titkos} Gy. P. Geh\'er, and T. Titkos, \emph{A characterisation of isometries with respect to the Lévy-Prokhorov
metric,} Annali della Scuola Normale Superiore di Pisa - Classe di Scienze, Vol. XIX (2019), 655--677.

\bibitem{gtv} Gy. P. Geh\'er, T. Titkos, D. Virosztek, \emph{On isometric embeddings of Wasserstein spaces -- the discrete case}, J. Math. Anal. Appl., to appear, arXiv: 1809.01101

\bibitem{hairer}
M. Hairer, J.C. Mattingly, M Scheutzow, \emph{Asymptotic coupling and a general form of Harris'theorem with applications to stochastic delay equations} Probab. Theory Related Fields  {149} (2011), 223--259.

\bibitem{navier-stokes}
M. Hairer, J.C. Mattingly, \emph{Spectral gaps in Wasserstein distances and the 2D stochastic Navier--Stokes equations} Ann. Probab. {36}(6) (2008), 2050--2091.


\bibitem{Kloeckner-2010} B. Kloeckner, \emph{A geometric study of Wasserstein spaces: Euclidean spaces,} Annali della Scuola Normale Superiore di Pisa - Classe di Scienze { IX, 2} (2010), 297--323.

\bibitem{kloeckner-hausdorff}
B. Kloeckner, \emph{A generalization of Hausdorff dimension applied to Hilbert cubes and Wasserstein spaces,} J. Topol. Anal. {4}(2) (2012), 203--235.

\bibitem{ultrametric} B. Kloeckner, \emph{A geometric study of Wasserstein spaces: Ultrametrics}  Mathematika {61} (2015), 162--178.

\bibitem{Lamperti}
J. Lamperti,
\emph{On the isometries of certain function-spaces}
Pacific J. Math. 8 (1958) 459--466. 


\bibitem{LV}
J. Lott, C. Villani, \emph{Ricci curvature for metric-measure spaces via optimal transport}, Ann. of Math. 169 (2009), 903--991.

\bibitem{Mankiewicz}
P. Mankiewicz, 
\emph{On extension of isometries in normed linear spaces}, 
{Bull. Acad. Polon. Sci. S\'er. Sci. Math. Astronom. Phys.} {20} (1972), 367--371.

\bibitem{molnar-levy} L. Moln\'ar, \emph{L\'evy isometries of the space of probability distribution functions,} J. Math. Anal. Appl. {380} (2011), 847--852.



\bibitem{VRS}
M.-K. von Renesse and K.-T. Sturm, Transport inequalities, gradient estimates,
entropy, and Ricci curvature, Comm. Pure Appl. Math. 58 (2005), no. 7, p. 923--940.

\bibitem{m2}
S. Srivastava, C. Li, D.B. Dunson, \emph{Scalable Bayes via Barycenter in Wasserstein Space}, Proccedings of Machine Learning Research {19}(8) 2018, 1--35. 2018.

\bibitem{S}
K.-T. Sturm, On the geometry of metric measure spaces. I, II, Acta Math. 196
(2006), no. 1, 65--131 and 133--177.

\bibitem{vallender} S. S. Vallender, \emph{Calculation of the Wasserstein distance between probability distributions on the line,} Theory Probab. Appl. {18} (1973), 784--786.

\bibitem{villani-book} C. Villani, \emph{Optimal Transport: Old and New,} (Grundlehren der mathematischen Wissenschaften)
Springer, 2009.

\bibitem{villani-ams-book} C. Villani, \emph{Topics in optimal transportation,} Graduate studies in Mathematics vol. 58, American Mathemtical Society, Providence, RI, 2003.

\bibitem{virosztek-asm-szeg} D. Virosztek, \emph{Maps on probability measures preserving certain distances --- a survey and some new results,} Acta Sci. Math. (Szeged) {84} (2018), 65--80.

\bibitem{WW}
J.H. Wells, L.R. Williams, 
Embeddings and extensions in analysis,
Ergebnisse der Mathematik und ihrer Grenzgebiete, Band 84. Springer-Verlag, New York-Heidelberg, 1975.

\end{thebibliography}
\end{document}